\documentclass[]{article}

\usepackage[utf8]{inputenc}
\usepackage[english]{babel}
\usepackage{pdfpages}
\usepackage[a4paper]{geometry}
\usepackage{graphicx}
\usepackage{amsmath,amsthm,amssymb}
\usepackage{mathtools}
\usepackage{wasysym}
\usepackage{microtype}
\usepackage[colorlinks=false, pdfborder={0 0 0}]{hyperref}
\usepackage[capitalise]{cleveref}

\usepackage{color}

\title{Measure equivalence and sofic approximations}
\author{Thiebout Delabie, Juhani Koivisto and Romain Tessera}

\date{\today}

\newtheorem{theorem}{Theorem}[section]
\newtheorem{lemma}[theorem]{Lemma}
\newtheorem{clai}[theorem]{Claim}
\newtheorem{proposition}[theorem]{Proposition}

\theoremstyle{definition}
\newtheorem{definition}[theorem]{Definition}
\newtheorem{remark}[theorem]{Remark}
\newtheorem{warning}[theorem]{Warning}

\newtheorem*{ack}{Acknowledgments}

\newcommand{\diam}{\operatorname{diam}}

\newcommand{\GL}{\operatorname{GL}}

\newcommand{\SOL}{\operatorname{SOL}}

\renewcommand{\Im}{\operatorname{Im}}

\newcommand{\supp}{\operatorname{supp}}

\newcommand{\SL}{\operatorname{SL}}
\newcommand{\LL}{\operatorname{L}}

\newcommand{\N}{\mathbb{N}}
\newcommand{\PP}{\mathbb{P}}
\newcommand{\Z}{\mathbb{Z}}

\newcommand{\R}{\mathbb{R}}

\newcommand{\B}{\mathfrak{B}}
\newcommand{\A}{\mathfrak{A}}

\newcommand{\uu}{\mathfrak{u}}
\newcommand{\vv}{\mathfrak{v}}
\newcommand{\aaa}{\mathfrak{a}}
\newcommand{\cc}{\mathfrak{c}}
\newcommand{\eps}{\varepsilon}

\newcommand{\Lamp}{\mathcal{L}}
\makeatletter
\newcommand{\AlignRight}[1]{\ifmeasuring@#1\else\omit\hfill$\displaystyle#1$\fi\ignorespaces}
\makeatother

\begin{document}

\maketitle

\begin{abstract}
We introduce a technique for producing a measure coupling between two sofic groups from a family of maps between their sofic approximations. We exploit this to construct measure couplings between pairs of groups with prescribed integrability conditions. As an application we show that solvable Baumslag-Solitar groups, Lamplighters and the group SOL are all exponentially measure equivalent to one another: in particular they are $L^p$ measure equivalent for all $p<\infty$. This is in sharp contrast with the fact that these groups are in general not quasi-isometric to one another: indeed, for instance $\Z/2\Z\wr \Z$ is not quasi-isometric to $\Z/3\Z\wr \Z$. 
\end{abstract}

\tableofcontents

\section{Introduction}

Measure equivalence between countable groups has been introduced by Gromov as a measure analogue of the notion of quasi-isometry. Recall that two countable groups $G$ and $H$ are measure equivalent (ME) if there exists a standard measure space $(\Omega,\mu)$ equipped with free, commuting, measure preserving actions of $G$ and $H$, such that both actions admit fundamental domains with finite measure.
A well-known consequence of a celebrated theorem of Ornstein-Weiss \cite{OrnWei-80} is that all infinite countable amenable groups are measure equivalent. 

As we shall recall in the next paragraph, the data of the two fundamental domains can be used to define two  cocycles. These can somehow be viewed as measured analogues of mutual morphisms between the groups, which are in a weak sense inverses of one another. It turns out that some rigidity can be restored by imposing integrability conditions on these cocycles.  For instance Bowen proved that $L^1$ ME preserves the volume growth \cite{Aus-16}. 
 A general theory of ``quantitative'' ME was developed in \cite{DKLMT-22} (see next paragraph).  This theory raises two main problems: 
\begin{itemize}
\item  finding limitations on the integrability properties of a measure equivalence coupling between two groups;
\item  conversely constructing measure equivalence couplings between those groups whose integrability properties match these limitations.
\end{itemize}
In  \cite{DKLMT-22}, it is shown that the volume growth and more importantly the isoperimetric profile can both be used to answer the first question (for amenable groups). Regarding the second problem, 
it could be tempting to exploit  the proof of Ornstein-Weiss's theorem. However, their result happens to be too general to produce measure equivalences with good integrability conditions. To remedy for this, the authors of \cite{DKLMT-22} exploit a notion of F\o lner tiling sequences\footnote{F\o lner tiling sequences appear in \cite{Dan-16} in a different context and under a different name (see also \cite{CecCor-19}).} to produce explicit ME couplings between amenable groups with prescribed integrability conditions. This method has then been used extensively in \cite{Escalier1} to prove that the constraint given by the isoperimetric profile is nearly sharp in a wide class of situations (see \S \ref{sec:motivation}). 
More recently, in \cite{LpMEnilp} Llosa and the first and third author use F\o lner tiling sequences in nilpotent groups to classify them up to $L^p$ ME for $p\leq 1$.

Despite their undeniable usefulness, F\o lner tiling sequences present several limitations. First, they  might not exist in all amenable groups. Besides, even when they do, they might not yield ME couplings with sharp integrability conditions. Finally, F\o lner tiling sequences are limited to amenable groups. 

The goal of this note to present a more general framework to produce various forms of couplings: measure equivalent, orbit equivalent but also a less known form of coupling called  ``measure subgroup couplings''.
Besides, our method applies to the wider class of sofic groups.

\subsection{Quantitative measure equivalence}\label{secIntro:Quant}
Before stating our results let us briefly recall what we mean by quantitative measure equivalence (more details are provided in \S \ref{sec:PrelimME}). In the sequel, we let $(G,S_G)$ and $(H,S_H)$ be finitely generated groups equipped with finite generating sets. 
Let $(\Omega, \mu)$ be a standard Borel $\sigma$-finite measure space equipped with commuting, free measure preserving actions of  $G$ and  $H$. We assume that these actions admit fundamental domains $X_G$ and $X_H$ such that $X_H$ has  measure $1$. Without more assumptions on $X_G$, we call such a coupling a {\bf measure subgroup} coupling\footnote{Note that if $G$ is a subgroup of $H$, then $\Omega=H$ equipped with its counting measure is trivially a measure subgroup coupling from $G$ to $H$.} from $G$ to $H$. If we assume that $X_G$ also has finite measure, then this is a {\bf measure equivalence} (ME for short) coupling. Our coupling will be called an {\bf orbit equivalent} (OE) coupling if in addition $X_G=X_H$ (for the connection with orbit equivalent p.m.p.\ actions, see \S \ref{sec:PrelimME}).

 A measure subgroup coupling comes with a cocycle $\alpha:G\times X_H\to H$, where $\alpha(g,x)\in H$ is the unique element which maps $g* x\in \Omega$ back to $X_H$. Similarly for an ME coupling, we define a cocycle $\beta: H\times X_G\to G$. 
 \begin{definition}\label{defIntro:QuantME}
     Let $\varphi, \psi \colon \mathbb{R}_{\geq 0}\to \mathbb{R}_{\geq 0}\cup\{\infty\}$ be non-decreasing functions and let $S_G\subset G$, $S_H\subset H$ be finite generating sets for $G$ and $H$. We will say that a measure subgroup coupling from $G$ to $H$ is \emph{$\varphi$-integrable} if there is $\delta>0$ such that for all $s\in S_G$
     \[
        \int_{X_H} \varphi\left(\delta|\alpha(s,x)|_{S_H}\right) d\mu(x) < +\infty.
     \]
     We will say that the ME coupling from $G$ to $H$ is \emph{$(\varphi,\psi)$-integrable} if, in addition,  there is $\delta'>0$ such that for all $s\in S_H$
     \[
        \int_{X_G} \psi\left(\delta'|\beta(s,x)|_{S_G}\right) d\mu(x) < +\infty.        
     \]
\end{definition}
When $\varphi(t)=t^p$, we generally write $L^p$-integrable instead of $\varphi$-integrable. Note that the extreme cases $p=0$ or $p=\infty$ both make sense here: 
$p=0$, when we do not impose any condition, and $p=\infty$, which means that $x\mapsto |\alpha(s,x)|_{S_H}$ is essentially bounded.
 For instance an ME coupling where we only impose a condition of $\varphi$-integrablity on the cocycle from $G$ to $H$ will be called $(\varphi,L^0)$-integrable. 
 
Note that we introduce the constants $\delta$ and $\delta'$ so that the definitions do not depend on a choice of generating sets.

Finally, to be coherent with the literature, two groups are called $L^p$ ME if there exists an $(L^p,L^p)$-integrable measure coupling between them.

 \subsection{Motivation: sharpness of the isoperimetric obstruction}\label{sec:motivation}

As already mentioned,  an efficient tool to restrain the integrability of couplings between amenable groups is their isoperimetric profile.
 Recall that given a finitely generated group $G$ equipped with a finite generating symmetric subset $S$, one defines its isoperimetric profile	\[j_{G}(n)=\sup_{|A|\leq n} \frac{|A|}{|\partial A|},\]
where $\partial A=S_\Gamma A\bigtriangleup A$. In \cite[Theorem 1]{DKLMT-22}, the following result is proved.
	\begin{theorem}\label{thmintroIsop}
		Assume that there exists either a $(\varphi,L^0)$-integrable ME-coupling, or an injective\footnote{Here injective means that we require that for a.e.\ $x\in X_H$, the map $\alpha(\cdot, x):G\to H$ is injective.} $\varphi$-integrable measure subgroup coupling from $G$ to $H$.  Then 
		\begin{itemize}
			\item if $\varphi(t)=t$ then \[j_{G}\succcurlyeq j_{H};\]
			\item if $\varphi$ and $t\mapsto t/\varphi(t)$ are increasing then
			\[
			j_{G}\succcurlyeq \varphi\circ j_{H}.
			\]
		\end{itemize}
	\end{theorem}
We refer to \cite{DKLMT-22} for concrete examples and references therein. Let us simply mention that this theorem implies that an $(L^p,L^0)$-integrable ME-coupling from $\Z^q$ to $\Z^k$, with $q>k$ must satisfy $p\leq k/q$, or that a  $(\varphi,L^0)$-integrable ME-coupling from $\Z/2\Z\wr \Z$ to $\Z$ must satisfy $\uu(t)=O(\log t).$ 
Using a clever idea, Correia recently managed to strengthen the conclusion by treating the ``critical case'' (see \cite[Theorem B]{Correia-24} for a precise statement). In particular, for the two examples above, his result yields strict conditions: i.e.\ respectively $p> k/q$, and $\uu(t_n)=o(\log t_n)$ along a subsequence.

\cite[Theorem 1.7]{Escalier1}, Escalier shows that Theorem \ref{thmintroIsop} is nearly optimal\footnote{up to a $\log$-factor.} when $H=\Z$ (whose profile grows linearly). For that purpose, she uses the family of groups constructed by Brieussel and Zheng \cite{BrieusselZheng}, which are designed to roughly realize any possible behavior of the isoperimetric profile. In the process she gives an almost full solution to the inverse problem of finding groups with prescribed integrability conditions for their OE-couplings with $\Z$.

More recently in \cite[Theorem~1.3]{Escalier2}, she manages to replace $\Z$ as target group by the lamplighter group $\Z/2\Z\wr \Z$ (whose profile grows as $\log n$), but at the cost of having to replace OE couplings with injective measure subgroup couplings. 

While in her first construction, she uses F\o lner tiling sequences, in the second one, she relies on results presented in this work (in her situation, the sofic approximations are F\o lner sequences).

\subsection{Sofic approximations}
In \cite{Das}, Das proves the following statement. 
\begin{theorem}
Let $G$ and $H$ be  two residually finite finitely generated groups $G$ and $H$, and let $(G_n)$ and $(H_n)$ be two sequences of finite quotients converging respectively to $G$ and $H$ for the marked groups topology. If $G_n$ and $H_n$  are uniformly quasi-isometric, then $G$ and $H$ are $L^\infty$ measure equivalent.
\end{theorem}
This served as a starting point for this work. Here we will generalize his construction to the situation where we have a certain sequence of maps $\uu_n$ (not necessarily uniform quasi-isometries) between sofic approximations of $G$ and $H$. 
 Recall that sofic approximations are finite graphs that approximate the Cayley graph of a group in the following sense.
\begin{definition}
	Let $G$ be a finitely generated group with a finite generating set $S_G$ and let $(\mathcal{G}_n)_n$ be a sequence of directed, labeled graphs. Then $(\mathcal{G}_n)_n$ is a sofic approximation of $(G,S_G)$ if for every $r>0$, the set $\mathcal{G}_n^{(r)}$ of vertices $x\in \mathcal{G}_n$ such that $B_{\mathcal{G}_n}(x,r)$ is  isomorphic to $B_G(e_G,r)$ as directed, labeled graphs satisfies \[\lim_{n\to \infty}\frac{\# \mathcal{G}_n^{(r)}}{\# \mathcal{G}_n}=1.\] 
\end{definition}
The most classical examples of sofic approximations are F\o lner sequences in Cayley graphs of amenable groups, or as in Das's theorem Cayley graphs of finite quotients of a residually finite group $G$ associated to a fixed generating set.

\

Note that there exists a right ``almost-action'' of $G$ on the vertex sets of the graphs $\mathcal{G}_n$. That is, for $g\in G$ and $x\in \mathcal{G}_n$ if moreover $x\in \mathcal{G}_n^{(|g|)}$, then there exists a unique isomorphism of labelled graph $\theta_x\colon B_G(e_G,|g|)\to B_{\mathcal{G}_n}(x,|g|)$ and we let  $xg: = \theta_x(g)$. 

We will equip $\mathcal{G}_n$ with the uniform probability measure  $\PP_{\mathcal{G}_n}$. The soficity assumption can be reformulated as 
\begin{equation}\label{eq:sofic}
\lim_n \PP_{\mathcal{G}_n}\left(\mathcal{G}_n^{(r)}\right)=1,
\end{equation}
for all $r\geq 0$.
\begin{warning}
We shall use (\ref{eq:sofic}) to allow ourselves to write expressions of the form:
\[\lim_n  \PP_{\mathcal{G}_n}\left(\left\{x\in \mathcal{G}_n\mid F(x,xg)\right\}\right)= \textnormal{something},\]
where $F$ is formula involving $x$ and $xg$, when strictly speaking we should be writing instead
\[\lim_n  \PP_{\mathcal{G}_n}\left(\left\{x\in \mathcal{G}_n^{(|g|)}\mid F(x,xg)\right\}\right)= \textnormal{something}.\]
\end{warning}


\subsection{Main results}

Our general setup is the following:
\begin{itemize}
\item Measure subgroup context: let $G$ and $H$ be two finitely generated groups, $(\mathcal{G}_n)$ be a sofic approximation of $G$, and $\uu_n:\mathcal{G}_n\to H$ a sequence of maps.
 \item Measure equivalence context: let $G$ and $H$ be two finitely generated groups, $(\mathcal{G}_n)$, $(\mathcal{H}_n)_n$ are sofic approximations of $G$ and $H$, and  $\uu_n:\mathcal{G}_n \to \mathcal{H}_n$ a sequence of maps.
\end{itemize}
 
 We will prove that under certain conditions on the sequence $\uu_n$, one can construct either a measure subgroup, an ME or an OE coupling from $G$ to $H$. By making these conditions ``quantitative'', we will furthermore obtain couplings with prescribed integrability properties.

 We view each $\mathcal{G}_n$ as a finite metric measure space, where the metric comes from the graph structure, and the measure is  $\PP_{\mathcal{G}_n}$. 
The conditions we shall impose on $\uu_n$ will consist in  ``relaxing''  in a probalistic sense coarse metric properties. By ``relaxing'', we mean that we will require that these conditions are satisfied with high probability as $n\to \infty$. 
For precise definitions, we refer to \S \ref{sec:statistical}, and more precisely to Definitions \ref{def:statisticSet} and \ref{def:statisticCoarseMetric}. Let us mention that these definitions require the use of a non-principal ultrafilter $\mathcal U$ on $\N$.


Our first statements are qualitative: they provide conditions on the maps $\uu_n$ such that there exist subgroup, ME or OE couplings. We shall say that a measure subgroup coupling is injective if  the map  $\alpha(\cdot,x):G\to H$ is injective for a.e.\ $x\in X_H$ (this property is required in \cref{thmintroIsop}).

\begin{theorem}\label{thm:Subgroup}
Assume we are given a sequence of injective maps $\uu_n\colon \mathcal{G}_n\to H$, which is $\mathcal U$-statistically Lipschitz. Then there exists an injective measure subgroup coupling $(\Omega,\mu)$ from $G$ to $H$.
\end{theorem}
Note that the cocycle $\alpha$ can be viewed as a measurable family of maps $\alpha(\cdot,x):G\to H$ indexed by the probability space $(X_H,\mu)$. Similarly, the maps $\uu_n$ yield a family of maps from larger and larger balls of $G$ to $H$, indexed by $\mathcal{G}_n$. Indeed, since $\mathcal{G}_n$ is a sofic approximation of $G$, if $r$ is fixed while $n$ tends to infinity, most of the balls $B_{\mathcal{G}_n}(x,r)$ are isomorphic to the ball $B_G(e_G,r)$, so that the maps $\uu_n$ provide us with a family of maps from $B(e_G,r)$ to $H$, indexed by $\mathcal{G}_n^{(r)}$. The idea behind our construction is to define the maps $\alpha(\cdot, x)$ as the ultralimit of these maps\footnote{More precisely the limit corresponds to the family of maps $g\mapsto \alpha(g^{-1},x)^{-1}$, see \cref{prop:SubgroupCoupling}.}, while $(X_H,\mu)$ is the utralimit of the sequence of probability spaces $(\mathcal{G}_n,\PP_{\mathcal{G}_n})$. The precise construction is explained in \S \ref{sec:approxCocycle}. 

We now address the case of a measure coupling, obtained from a sequence of maps   $\uu_n\colon \mathcal{G}_n\to \mathcal{H}_n$. Roughly the idea is analogous to the previous situation: we exploit the fact that $\mathcal{H}_n$ is a sofic approximation of $H$ to view the maps $\uu_n$ as a family of maps from $B_G(e_G,r)$ to $H$ (for every $r$ and as $n$ tends to infinity). But the details here are trickier (see \cref{prop:approximatecocycle} for a precise formulation).

In the following theorem, we specify an important property of the resulting coupling: {\bf coboundedness}.  An ME coupling from $G$ to $H$ is  cobounded if up to measure zero the fundamental domain of $G$ is contained in finitely many translates of the fundamental domain of $H$. The coupling is called mutually cobounded if it is cobounded both from $G$ to $H$, and from $H$ to $G$. For instance an OE coupling is mutually cobounded.   We will come back to this condition in the next subsection. Let us simply mention that coboundedness is somehow independent of integrability conditions. Indeed, although an integrability condition for the cocycle from $G$ to $H$ relies on a specific choice of fundamental domain for $H$, it does not depend on how the two fundamental domains behave with respect to each other.

\begin{theorem}\label{thm:ME}
Assume that  the sequence of maps $\uu_n\colon \mathcal{G}_n\to \mathcal{H}_n$ is $\mathcal U$-statistically a coarse equivalence. Then there exists an ME coupling from $G$ to $H$. Moreover,
\begin{itemize}
\item if there exists $C>0$ such that every $\uu_n$ has $C$-dense image, then the coupling is cobounded from $G$ to $H$; 
\item  if there exists $C>0$ such that every $\uu_n$ has point-preimages of diameter at most $C$, then the coupling is cobounded from $H$ to $G$;
\item if the maps $\uu_n$ are bijective, then the coupling is an OE coupling.
 \end{itemize}
\end{theorem}

We now turn to integrability conditions. Let us go back to our heuristic description of how the cocycle $\alpha$ is obtained from our sequence of maps $\uu_n$.  In order to guess the connection between properties of $\uu_n$ and integrability of the cocyle $\alpha$, we can follow the rough description given above. Note that requiring that the cocycle $\alpha$ is $L^{\infty}$ implies that the family of maps $\alpha(\cdot,x):G\to H$ are uniformly Lipschitz\footnote{for right-invariant word metrics on both groups.}. By extension, imposing an integrability condition on $\alpha$ can be viewed as a way to tame the defect of being Lipschitz for this family of maps. It is therefore expected to be obtained from an analogous summability condition on the sequence  of its approximations. Precisely, we get the following statement.

\begin{theorem}\label{thm:quantSubgroup}
	Let $\varphi\colon \R^+\to \R^+$ be an unbounded non-decreasing map. Assume the maps $\uu_n\colon \mathcal{G}_n\to H$ are injective and satisfy the following:
 for every $s\in S_G$ there exists $\delta>0$ such that
\begin{equation}\label{eq:quantLip}
	\sum_{r=0}^\infty \varphi\!\left(\delta r\right) \lim_\mathcal U\PP_{\mathcal{G}_n}\left(\left\{x\in \mathcal{G}_n\mid d_{\mathcal{H}_n}(\uu_n(x),\uu_n(xs))=r\right\}\right)<\infty;
	\end{equation}
	Then, there exists an injective measure subgroup coupling from $G$ to $H$ which is $\varphi$-integrable. 
\end{theorem}

We now turn to the ME case.  
In order to obtain integrability conditions on the inverse cocycle, the ideal situation is when the $\uu_n$ are bijections, as we can simply treat $\beta$ as we did for $\alpha$ but replacing $\uu_n$ by their inverses. By contrast with our qualitative statement \cref{thm:ME}, we shall impose that our maps are not far from being bijective by assuming a uniform bound on the diameter of point-preimages and on the coarse density of the image. 
\begin{theorem}\label{thm:quantME}
	Let $\varphi,\psi\colon \R^+\to \R^+$ be two unbounded non-decreasing maps. Let $G$ and $H$ be two finitely generated groups with sofic approximations $(\mathcal{G}_n)_n$ and $(\mathcal{H}_n)_n$ and $\uu_n\colon \mathcal{G}_n\to \mathcal{H}_n$ satisfy the following:
	\begin{enumerate}
	\item There exists $C>0$ such that every $\uu_n$ has $C$-dense image and has point-preimages of diameter at most $C$;
	\item For every $s\in S_G$ there exists $\delta>0$ such that (\ref{eq:quantLip}) holds.
	\item For every $h\in H$ of length at most $2C+1$, there exists $\delta>0$ such that
		\begin{equation}\label{eq:quantExp}
	\sum_{r=0}^\infty \psi\!\left(\delta r\right)  \lim_\mathcal U \PP_{\mathcal{H}_n}\left(\left\{y \in \uu_n(\mathcal{G}_n) \mid  \diam_{\mathcal{G}_n}(\uu_n^{-1}(y)\cup \uu_n^{-1}(yh))=r\right\}\right)<\infty.
		\end{equation}
	\end{enumerate}
	Then, there exists a  measure equivalence coupling from $G$ to $H$ which  is $(\varphi,\psi)$-integrable and mutually cobounded (OE if the maps $\uu_n$ are bijective).\end{theorem}

To understand the formula (\ref{eq:quantExp}) above, observe that when the map $\uu_n$ is not invertible, a natural substitute for $d_{\mathcal{G}_n}(\uu_n^{-1}(y),\uu_n^{-1}(yh))$ is  $\diam_{\mathcal{G}_n}(\uu_n^{-1}(y)\cup \uu_n^{-1}(yh))$.

\subsection{An application}
We now present an application to a family of solvable groups which has been intensively studied from the point of view of quasi-isometries.  Denote by $\mathcal M$ the following class of groups:
\begin{itemize}
\item for $k\geq 2$,  the lamplighter group $\Lamp_k=(\Z/k\Z)\wr\Z$;
\item for $k\geq 2$,  Baumslag-Solitar's group $\operatorname{BS}(1,k)=\Z[1/k]\rtimes \Z$, where $n\in \Z$ acts by multiplication by $k^n$; 
\item $\SOL_A=\Z^2\rtimes_A\Z$, where $A\in \SL_2(\Z)$ has a real eigenvalue $\lambda>1$.
\end{itemize}

These groups share important geometric properties. For instance, all groups from $\mathcal M$ have isometric asymptotic cones \cite{Cornulier_dimcone}, and same isoperimetric profile \cite{coulhonGeometricApproachOndiagonal2001} (see also \cite{Tes-16} for a common description as uniform lattices in $\SOL$-type groups). 

On the other hand, they are also known for their strong quasi-isometric rigidity features. For instance that $\Lamp_k$ is quasi-isometric to $\Lamp_{k'}$ if and only if $k$ and $k'$ are powers of some common integer \cite{EFWII}. The same statement holds for $\operatorname{BS}(1,k)$  and  $\operatorname{BS}(1,k')$ \cite{MR1608595}. By contrast, $\SOL_A$ is a uniform lattice in the group $\R^2\rtimes \R$, where $\R$ acts by a diagonal matrix with coefficients $(e^t,e^{-t})$, hence the quasi-isometry class of $\SOL_A$ does not depend on $A$. Moreover, groups from two of these three classes cannot be quasi-isometric: this follows for instance from the fact that $\SOL$, $\operatorname{BS}$ and $\Lamp$ have respectively asymptotic dimension $3$, $2$ and $1$. Another criterion is finite presentability: while $\operatorname{BS}$ and $\SOL$ are finitely presented, $\Lamp$ is not.

All these properties make the class $\mathcal M$ particularly interesting to study from the quantitative ME point of view. 
In fact, these rigidity results can be interpreted in terms of quantitative ME. Indeed, by \cite{Sau-06,Sha-04}, any two finitely generated amenable groups are quasi-isometric if and only if they admit an $L^{\infty}$ ME coupling which is mutually cobounded.

We recall that $\varphi$-integrability of a cocycle becomes more restrictive as $\varphi$ goes faster to infinity. In particular, exponential integrability is stronger than being $L^p$ for all $p<\infty$. Contrary to $L^p$ ME for $p\geq 1$, exponentially integrable ME does not behave well under composition of couplings, and therefore may not define an equivalence relation. In \cite{DKLMT-22}, the following stronger property is introduced to deal with this defect.

\begin{definition}
Let $G$ and $H$ be two finitely generated groups, 
$X$ be a finite measure space equipped with a measure preserving action of $G$, and  $c\colon G\times X\to H$ be a cocycle. The cocycle $c$ is strongly exponentially integrable (denoted $\exp^\diamond$-integrable), if for every $\varepsilon>0$ there exists a $\delta>0$ and $C>0$ such that
	\[\int_{X}\exp\left(\delta|c(g,x)|\right)d\mu(x) \le C\exp(\varepsilon |g|)\]
	for every $g\in G$.
\end{definition}
The main interest of this notion lies in the fact that being (mutually coboundedly) strongly exponentially ME is an equivalence relation\footnote{Mutual coboundedness is not addressed in this reference, but it is easy to check that the composition of couplings which is described there do preserve this condition.} among finitely generated groups \cite[Proposition 2.29.]{DKLMT-22}.

 In \cite[Theorem 8.1.]{DKLMT-22} an explicit OE coupling is constructed from $\Lamp_k$ to $\operatorname{BS}(1,k)$ which is $(L^\infty,\exp^\diamond)$. This construction strongly relies on the fact that the parameter $k$ is the same for both groups: it is a priori much harder to construct a coupling between $\Lamp_2$ and $\Lamp_3$ for instance. 
Here we use a variant of \cref{thm:quantME} to construct a mutually cobounded $(L^{\infty},\exp^\diamond)$ ME coupling from $\Lamp_k$ to $\SOL_A$ (see \cref{thm:example}). As a corollary, we obtain the following theorem, which contrasts with the situation for quasi-isometries: 

\begin{theorem}\label{thm:SolvableExpME}
There exists a mutually cobounded $(\exp^\diamond,\exp^\diamond)$-integrable measure equivalence coupling between any two groups in $\mathcal M$.
\end{theorem}

Let us end this section with a problem. The previous theorem implies in particular that amenable Baumslag-Solitar groups $\operatorname{BS}(1,n)$, for $n\geq 2$ are $\exp^\diamond$-ME. 
What can be said of the class of {\it non-amenable} Baumslag-Solitar groups $\operatorname{BS}(m,n)$, with $1<m<n$? 
By contrast with the amenable case, Whyte has famously proved that any two groups from this class are quasi-isometric \cite{Whyte01}.
Recently, Gaboriau, Poulin, Tserunyan, Tucker-Drob, and Wrobel have announced a proof of the measurable analogue of this result, namely that any two groups from this class are ME. Can this be generalized to $L^p$ ME, or even $\exp^\diamond$-ME?

\subsection{Organisation}
In \S \ref{sec:PrelimME}, we recall some basic facts about ME and OE couplings, and their quantitative versions. Subgroup couplings are not addressed there, as the definitions given in the introduction are sufficient for our purpose. While the use of a (non-principal) utrafilter $\mathcal U$ is omnipresent in this paper, the only non-elementary statement we need about it is Loeb's construction of an ultralimit of probability spaces, which is recalled in \S \ref{sec:Loeb}.
In \S \ref{sec:StatisticCoarse} we introduce the concept of $\mathcal U$-statistical properties of a sequence of maps between finite metric spaces (equipped with the uniform probability measure). 
The main part of the paper is \S \ref{BoxSpaces}, where the various couplings are constructed (the plan of this section is explained there). 
In \S \ref{sec:integrability}, we treat the integrability properties of our couplings. In \S \ref{sec:Proofs}, we gather ingredients from the two previous sections to prove the main theorems of the introduction.   
Finally in \S \ref{sec:M}, we prove the theorems relative to the class $\mathcal M$.

\begin{ack} We thank Alessandro Carderi for helping us with the references on Loeb's theorem. We are grateful to him and François Le Maître for many useful conversations around this 
	project. We are indepted to Corentin Correia for his careful reading, and we thank Claudio Llosa Isenrich, and Matthieu Joseph for their useful remarks.
\end{ack}

\section{Preminaries}
\subsection{Variations on (quantitative) measure and orbit equivalence}\label{sec:PrelimME}
We begin by recalling the notion of measure equivalence. 
For the sake of clarity, we first recall some terminology following \cite{MBS}. A \emph{standard Borel measure space} $(X,\mu)$ is a Borel space $(X,\mathcal{B}(X))$, frequently also known as a measurable space, with a measure $\mu$ on the $\sigma$-algebra $\mathcal{B}(X)$ given by the Borel $\sigma$-algebra of some Polish (separable and completely metrizable) topology on $X$. The elements of $\mathcal{B}(X)$ are called \emph{Borel}. A Borel subset $X_0 \subseteq X$ is \emph{conull} if $\mu(X \setminus X_0)=0$ and a property that holds for all $x \in X_0$ is said to hold \emph{almost everywhere} or \emph{a.e.} The space $(X,\mu)$ is said to be a standard Borel probability space if $\mu(X)=1$.

By a \emph{measure preserving action} of a finitely generated group $G$ on $(X,\mu)$, for short $G \curvearrowright (X,\mu)$, we mean that the action map $G \times X \rightarrow X$ given by $(g,x) \mapsto g\ast x$ is Borel and that $\mu(g\ast E) = \mu(E)$ for all $g \in G$ and all $E \in \mathcal{B}(X)$. A measure preserving action on a standard Borel probability space is said to be probability measure preserving. Finally, by a \emph{measure fundamental domain} for $G \curvearrowright (X,\mu)$ we mean a Borel set $X_G \subseteq X$ for which $\mu(g\ast X_G \cap X_G)= 0$ for all $g \in G \setminus \lbrace e_G\rbrace$ and $\mu(X \setminus G\ast X_G) = 0$. 
The definition of measure equivalence can now be stated as follows.
\begin{definition}
	Two infinite finitely generated groups $G$ and $H$ are said to be \emph{measure equivalent} if there exists a standard Borel measure space $(\Omega,\mu)$ of infinite measure with commuting measure preserving actions of $G$ and $H$ with measure fundamental domains of finite measure, that is, there exist Borel sets $X_G, X_H \subseteq \Omega$ of finite measure for which
	\begin{itemize}
		\item $\mu(g\ast X_G \cap X_G)= 0$ for all $g \in G \setminus \lbrace e_G\rbrace$ and $\mu(\Omega \setminus G\ast X_G) = 0$;
		\item $\mu(h\ast X_H \cap X_H)= 0$ for all $h \in H \setminus \lbrace e_H\rbrace$ and $\mu(\Omega \setminus H\ast X_H) = 0$.
	\end{itemize}
\end{definition}
The quadruple $(\Omega,X_G,X_H,\mu)$ is called a measure equivalence coupling from $G$ to $H$ and it is cobounded if the fundamental domains are contained in finitely many translates of each other.
\begin{remark}
	If $(\Omega,X_G,X_H,\mu)$ is a measure equivalence coupling and $\mu(X_G \triangle X_H)= 0$, then $X_G$ is a measure fundamental domain for $H$, and we can take $X_G$ as a common measure fundamental domain for both actions.
\end{remark}

To  quantify measure equivalence we need some terminology regarding integrability of \emph{cocycles}. Towards this, let $G$ be some finitely generated group with a measure preserving action on a standard Borel measure space $(X,\mu)$ with finite measure, and let $c \colon G \times X \rightarrow H$ be a Borel function into another finitely generated group $H$.
We say that $c \colon G \times X \rightarrow H$ is a \emph{Borel cocycle} if for each $g,g' \in G$, we have $c(gg',x)=c(g,g'\!.\hspace{1pt}x)c(g',x)$ for a.e. $x \in X$.
\begin{definition}
	Let $\varphi\colon \R^+\to \R^+$ be a non-decreasing maps, let $X$ be a finite measure space and let $G$ and $H$ be two groups with a cocycle $c\colon G\times X\to H$. The cocycle $c$ is $\varphi$-integrable if there exists a $\delta>0$ such that
	$$\int_{X_H}\varphi\left(\delta|c(g,x)|\right)d\mu(x)<+\infty$$
	for every $g\in G$ and $c$ is strongly $\varphi$-integrable (or $\varphi^\diamond$-integrable) if for every $\varepsilon>0$ there exists a $\delta>0$ and $C>0$ such that
	\[\int_{X_H}\varphi\left(\delta|c(g,x)|\right)d\mu(x) \le C\varphi(\varepsilon |g|)\]
	for every $g\in G$.
\end{definition}
We will write $\LL^p$-integrable when $\varphi(x)=x^p$ and write $\LL^\infty$ when $x\mapsto c(g,x)$ is essentially bounded for every $g\in G$.
Note that when $Y$ is of finite measure, then $\psi$-integrability implies $\varphi$-integrability, whenever there exists $K>0$ such that $\varphi(x)\le K\psi(K x)+K$.
\par

A measure equivalence $(\Omega,X_G,X_H,\mu)$ comes naturally equipped with two Borel cocycles given by its measure fundamental domains. These are given by the Borel maps $\alpha \colon G \times X_H \rightarrow H$ where for every $g \in G$, we have that $$g\ast x \in \alpha(g,x)^{-1}\ast X_H$$ for a.e. $x \in X_H$, and $\beta \colon H \times X_G \rightarrow G$ where for every $h \in H$, we have that $$h\ast y \in \beta(h,y)^{-1}\ast X_G$$ for a.e. $y \in X_G$. The map $\alpha \colon G \times X_H \rightarrow H$ is a Borel cocycle with respect to a measure preserving action $G \curvearrowright (X_H, \mu\vert_{X_H})$ where for every $g \in G$, we have that $$g \cdot x = \alpha(g,x)\ast g\ast x$$ for a.e. $x \in X_H$, and $\beta \colon H \times X_G \rightarrow G$ is a Borel cocycle with respect to a measure preserving action $H \curvearrowright (X_G, \mu\vert_{X_G})$ given by $$h \cdot y = \beta(h,y)\ast h\ast y$$ for a.e. $y \in X_G$.
The Borel cocycles defined in this way are called the \emph{measure equivalence cocycles}.
\begin{definition}
	Let $\varphi,\psi\colon \R^+\to \R^+$ be non-decreasing maps, let $G$ and $H$ be two groups with a measure equivalence coupling $(\Omega,X_G,X_H,\mu)$ from $G$ to $H$, let $\alpha\colon G\times X_H \to H$ be the cocycle such that $g*x \in \alpha(g,x)^{-1}X_H$ and let $\beta\colon H\times X_G \to G$ be the cocycle such that $h*x \in \beta(h,x)^{-1}X_G$.
	The measure equivalence coupling $(\Omega,X_G,X_H,\mu)$ from $G$ to $H$ is $(\varphi,\psi)$-integrable if $\alpha$ is $\varphi$-integrable and $\beta$ is $\psi$-integrable.
\end{definition}

We now turn our attention to orbit equivalence of groups, which implies measure equivalence. Towards this, we recall that a probability measure preserving action $G \curvearrowright (X,\nu)$ of a finitely generated group $G$ on a standard Borel probability space $(X,\nu)$ is \emph{essentially free} if there exists a $G$-invariant conull Borel subset of $X$ on which the action is free. The definition of orbit equivalence can now be stated as follows. 
\begin{definition}
	Two infinite finitely generated groups $G$ and $H$ are said to be \emph{orbit equivalent} if there exists a standard Borel probability space $(X,\nu)$ with essentially free probability measure preserving actions $G \curvearrowright (X,\nu)$ and $H \curvearrowright (X,\nu)$ and a conull invariant Borel subset $X_0 \subseteq X$ such that $G\ast x = H\ast x$ for all $x \in X_0$.
	The space $(X,\mu)$ is called an orbit equivalence from $G$ to $H$.
\end{definition}
Like measure equivalence, orbit equivalence defines two Borel cocycles, given by $\alpha' \colon G \times X \rightarrow H$ and $\beta' \colon H \times X \rightarrow G$, where for each $g \in G$, we have $g\ast x = \alpha'(g,x)\ast x$ for every $x \in X_0$, and for each $h \in H$, that $h\ast x = \beta'(h,x)\ast x$ for every $x \in X_0$. These are called the \emph{orbit equivalence cocycles}. 

As mentioned before, orbit equivalence implies measure equivalence. In fact, we have an equivalence, when we assume the fundamental domains in the measure equivalence coupling to be the same. In addition, the correspondent cocycles are the same.
\begin{proposition}[Folklore]
	Let $\varphi,\psi\colon \R^+\to \R^+$ be non-decreasing maps and let $G$ and $H$ be two infinite, finitely generated groups. There exists a measure equivalence coupling $(\Omega,X,X,\mu)$ from $G$ to $H$ that is $(\varphi,\psi)$-integrable if and only if there exists an orbit equivalence coupling $(X,\nu)$ from $G$ to $H$ that is $(\varphi,\psi)$-integrable.
\end{proposition}
\begin{proof}
	Let $(\Omega,X,X,\mu)$ be a measure equivalence coupling from $G$ to $H$. As previously, we write $\alpha \colon G \times X \rightarrow H$ and $\beta \colon H \times X \rightarrow G$ for the measure equivalence cocycles.
	
	First we need the action of $G\times H$ to be essentially free as well. In general this is not the case. Therefore we make the following construction.
	Let $\mathcal{B} = \lbrace 0,1\rbrace^{G \times H}$ and equip it with the shift action $G \times H \curvearrowright \mathcal{B}$. Considering each copy of $\lbrace 0, 1\rbrace$ with the equiprobability measure, we obtain a product measure $\mu'$ on $\mathcal{B}$ with respect to which the shift action is probability measure preserving and essentially free as the product group is infinite. Extending the action of $G$ and $H$ diagonally to $\Omega \times \mathcal{B}$ we now obtain a new measure equivalence $(\Omega \times \mathcal{B}, X \times \mathcal{B}, X \times \mathcal{B}, \mu \times \mu')$ for $G$ and $H$ and where the action of $G\times H$ is essentially free. Its measure equivalence cocycles are given by $\alpha'' \colon G \times (X \times \mathcal{B}) \rightarrow H$ where $\alpha''(g,(x,z)) = \alpha(g,x)$ and $\beta'' \colon H \times (X \times \mathcal{B}) \rightarrow G$ where $\beta''(h,(x,z)) = \beta(h,x)$ which are $\varphi$-integrable and $\psi$-integrable respectively. 
	
	Now, the induced actions of $G$ and $H$ on $(X \times \mathcal{B}, \mu_{X} \times \mu')$ are essentially free and probability measure preserving, where $\mu_{X} = \mu(X)^{-1}\mu\vert_{X}$.
	We claim that these induced actions are orbit equivalent, and that their cocyles are given by the measure equivalence cocycles, which proves that $(X\times \mathcal{B},\mu\times \mu')$ is a $(\varphi,\psi)$-integrable orbit equivalence coupling.
	First, we note that for all $g \in G$, we have $g\ast (x,z) \in \alpha(g,x)^{-1}\ast X \times \mathcal{B}$ for a.e. $(x,z) \in X \times \mathcal{B}$. Fixing $g \in G$ and letting $h = \alpha(g,x)$, we then have for the induced actions that $g \cdot (x,z) = g\ast \alpha(g,x)\ast (x,z)$ and $h \cdot (x,z) = \beta(\alpha(g,x),x)\ast \alpha(g,x)\ast (x,z)$ for a.e. $(x,z) \in X \times \mathcal{B}$.
	Now, without loss of generality it suffices to show that $\beta(\alpha(g,x),x)=g$ for a.e.\ $(x,z) \in X \times \mathcal{B}$, as it proves that $G\cdot x\subseteq H\cdot x$ and that its orbit equivalence cocycle $G\times (X\times \mathcal{B})\to H$ is the same as $\alpha''$.
	This is true since $\alpha(g,x)\ast (x,z) \in \beta(\alpha(g,x),x)^{-1}\ast X \times \mathcal{B}$ and $g\ast (x,z) \in \alpha(g,x)^{-1}\ast X_G \times \mathcal{B}$ for a.e. $(x,z) \in X \times \mathcal{B}$ and $X\times\mathcal{B}$ is a measure fundamental domain. Moreover, for every $g \in G$, we have $g \cdot (x,z) = \alpha''(g,(x,z)) \cdot (x,z)$ for almost every $(x,z) \in X \times \mathcal{B}$. This concludes the first implication.
	\\
	
	Towards proving the converse, let $(X,\nu)$ be an orbit equivalence coupling from $G$ to $H$ that is $(\varphi,\psi)$-integrable. Now, take $\Omega = X\times G$ with $\mu = \nu\times \bar{\mu}$, where $\bar{\mu}$ is the counting measure. Note that $X$ can be seen as the subset $X\times \{e_G\}$ of $\Omega$.
	Next, let $G\times H$ act on $(\Omega,\mu)$ such that $(g,h)\ast (x,g')= (h\cdot x , gg'\beta'(h,x)^{-1})$. Then $(\Omega,X,X,\mu)$ is a $(\varphi,\psi)$-integrable measure equivalence coupling.
	Indeed, $\Omega$ is standard Borel, as it is the product of two standard Borel spaces. The space $X$ is a measure fundamental domain of $G$, since $G$ acts by left multiplication on the $G$-component in $\Omega$. Towards proving that $X$ is also a fundamental domain for $H$, note that $h\ast (x,e_G)\in \beta'(h,x)^{-1}\ast X$ for every $h\in H$ and $x\in X$, note that $\beta'(h,x)=e_G$ if and only if $h=e_H$ for almost every $x\in X$ because the action of $H$ on $X$ is essentially free and note that $h\mapsto \beta'(h,x)$ is surjective for almost every $x\in X$. Hence, $X$ is a fundamental domain for the $H$-action and $\beta(h,(x,e_G))=\beta'(h,x)$. Finally, note that $(g, \alpha'(g,x))\ast (x,e_G) \in X$, since $\beta'(\alpha'(g,x),g\cdot x)=g$ for every $g\in G$ and almost every $x\in X$. So, $\alpha(g,(x,e_G))=\alpha'(g,x)$.
	Thus, $G$ and $H$ have commuting, measure preserving actions on $\Omega$, both actions have $X$ as a fundamental domain with finite measure and the cocycles $\alpha$ and $\beta$ are $\varphi$-integrable and $\psi$-integrable respectively.
	This concludes the proof.
\end{proof}

\subsection{Ultralimit of probability spaces}\label{sec:Loeb}
We briefly recall the Loeb construction of an ultralimit of probability spaces (see \cite{CKTD-13,Car-16} for more details).  This material will be used in \S \ref{section:non-standard}. In what follows, $\mathcal{U}$ denotes a non-principal ultrafilter on $\N$.

Let $\{(X_n,\nu_n)\}_{n\in\mathbb N}$ be a sequence of finite probability spaces. The ultraproduct $X_{\mathcal{U}}$ is defined as usual: we consider the equivalence relation $\sim_{\mathcal{U}}$ on $\prod_n X_n$ defined as $(x_n)_n\sim_{\mathcal{U}}(y_n)_n$ if $\{n\in \N\colon x_n=y_n\}\in {\mathcal{U}}$ and we set
\[X_{\mathcal{U}}\coloneqq \prod_{n\in\mathbb N}X_n/\sim_{\mathcal{U}}.\]

We will denote by $x_{\mathcal{U}}$ and $A_{\mathcal{U}}$ elements and subsets of $X_{\mathcal{U}}$. For a sequence $\{x_n\in X_n\}_n$, we will denote its class in $X_{\mathcal{U}}$ by $[x_n]_{\mathcal{U}}$ and similarly for a sequence of subsets $\{A_n\subset X_n\}_n$ we will denote its class by $[A_n]_{\mathcal{U}}$. If each $X_n$ is a probability space, then there is a canonical way to construct a probability measure on $X_{\mathcal{U}}$, the Loeb probability space, such that sets of the form $[A_n]_{\mathcal{U}}$ are measurable. 

\begin{theorem} 
Let $\{(X_n,\mu_n)\}_{n\in\N}$ be a sequence of finite probability spaces and let $X_{\mathcal{U}}$ be the ultraproduct of the sequence $\{X_n\}_n$. Then there exists a complete probability space $(X_{\mathcal{U}},\mathfrak{B}_{\mathcal{U}},\mu_{\mathcal{U}})$ such that
  \begin{enumerate}
  \item for every sequence $\{A_n\subseteq X_n\}_n$ of subsets, the set $[A_n]_{\mathcal{U}}$ is measurable and $\mu_{\mathcal{U}}([A_n]_{\mathcal{U}})=\lim_{\mathcal{U}}\mu_n(A_n)$,
  \item for every measurable subset $A_{\mathcal{U}}\subseteq X_{\mathcal{U}}$ there exists a sequence of subsets $\{A_n\subseteq X_n\}_n$ such that $\mu_{\mathcal{U}}(A_{\mathcal{U}}\Delta [A_n]_{\mathcal{U}})=0$.
  \end{enumerate}
\end{theorem}

\section{Statistical coarse geometry of finite metric spaces}\label{sec:StatisticCoarse}

This section is devoted to $\mathcal U$-statistical metric properties of a sequence of maps between finite metric spaces.   We believe that the interest of these definitions goes beyond the setting of this paper, and we encourage the interested reader to extend them even beyond the framework of finite metric spaces\footnote{A possible generalization of finite metric space adapted to these notions could be that of a random pointed metric spaces: a measurable family of pointed metric spaces indexed by a probability space.}.

\subsection{$\mathcal U$-statistical coarse properties}\label{sec:statistical}
We say that a sequence $(X_n)_n$ of finite metric spaces is \emph{uniformly locally finite} if for every $R>0$ there exists a $N_R>0$ such that the cardinality $\#B_{X_n}(x,R)\le N_R$ for every $x$ in any of the $X_n$ where $B_{X_n}(x,R)=\{y\in X_n\mid d_{X_n}(x,y)\le R\}$.
Given a finite set $X$, we denote by $\PP_X$ the uniform probability measure on $X$.

In what follows, we let $\mathcal{U}$ be a non-principal ultrafilter on $\N$. 
\begin{definition}\label{def:statisticSet}
	Let $(Y_n)_n$ and $(Z_n)_n$ be two sequences of metric spaces and $\uu_n\colon Y_n\to Z_n$ be a sequence of maps.
	\begin{itemize}
		\item Assume that $Y_n$ is finite. We say that the sequence $(\uu_n)_n$ is \emph{$\mathcal{U}$-statistically coarsely injective} if
		\[\lim_{C\to \infty}\lim_\mathcal{U}\PP_{Y_n}\left(\left\{y\in Y_n \mid \diam_{Y_n}(\uu_n^{-1}(\uu_n(y)))\geq C\right\}\right)=0.\]
	
		\item Assume that $Z_n$ is finite. We say that the sequence $(\uu_n)_n$ is \emph{$\mathcal{U}$-statistically coarsely surjective} if
		\[\lim_{C\to \infty}\lim_\mathcal{U}\PP_{Z_n}\left(\left\{z\in Z_n \mid d_{Z_n}(z,\uu_n(Y_n))\ge C\right\}\right)=0.\]
		\item Assume that $Y_n$ and $Z_n$ are finite.  The sequence $(\uu_n)_n$ is \emph{$\mathcal{U}$-statistically coarsely bijective} if it is both $\mathcal{U}$-statistically coarsely injective and $\mathcal{U}$-statistically coarsely surjective.
	\end{itemize}
\end{definition}
Using the material of \S \ref{sec:Loeb},  $\lim_\mathcal{U}\PP$ can be interpreted as a probability measure on a certain (non-standard) space. This allows to reformulate  the above notions as follows:  $\mathcal{U}$-statistical coarse injectivity means that the probability that the diameter of the preimages are larger than $C$ tends to zero as $C$ tends to infinity. Similarly, is $\mathcal{U}$-statistically coarsely surjective means that the probability that an element of the target space is at distance $C$ from the image tends to zero as $C$ tends to infinity.

\begin{proposition}\label{prop:compaireSize}
	Let $(Y_n)_n$ and $(Z_n)_n$ be two sequences of uniformly locally finite metric spaces and let $\uu_n\colon Y_n\to Z_n$ be a sequence of maps, then
	\begin{itemize}
		\item If the sequence $(\uu_n)_n$ is $\mathcal{U}$-statistically coarsely surjective, then $\lim_\mathcal{U}\# Z_n/\# Y_n$ is finite.
		\item If the sequence $(\uu_n)_n$ is $\mathcal{U}$-statistically coarsely injective, then  $\lim_\mathcal{U}\# Y_n/\# \Im\uu_n$ (so in particular $\lim_\mathcal{U}\# Y_n/\# Z_n$) is finite.
	\end{itemize} 
\end{proposition}
\begin{proof}
	When this sequence of maps is $\mathcal{U}$-statistically coarsely surjective, then there exists $C>0$ such that $\#\{z\in Z_n \mid d_{Z_n}(z,\uu_n(Y_n))\le C\}\ge \frac{1}{2}\#Z_n$. As $(Z_n)_n$ is uniformly locally finite, for every $C$ there exists a constant $N_C$ such that every ball in $Z_n$ of radius $C$ has cardinality at most $N_C$ for any $n$. Hence, $\frac{1}{2}\#Z_n\le N_C\#\uu_n(Y_n)\le N_C\#Y_n$. Therefore $\# Z_n/\# Y_n\le 2N_C$, which proves the first claim.
	
	When this sequence of maps is $\mathcal{U}$-statistically coarsely injective, then there exists $C>0$ such that \[\lim_\mathcal{U}\PP_{Y_n}\left(\{y\in Y_n \mid \diam_{Y_n}(\uu_n^{-1}(\uu_n(y)))\le C\}\right)\ge \frac{1}{2}.\] As $(Y_n)_n$ is uniformly locally finite, for every $C$ there exists a constant $N'_C$ such that every ball in $Y_n$ of radius $C$ has cardinality at most $N'_C$ for any $n$. Hence, $\frac{1}{2}\#Y_n\le N'_C\#\uu_n(Y_n)$. Therefore $\# Y_n/\#\uu_n(Y_n)\le 2N'_C$, which proves the second claim.
\end{proof}

We now turn to ``statistical versions'' of metric properties such as: Lipschitz, expansivity, coarse equivalence... 

\begin{definition}\label{def:statisticCoarseMetric}
	Let $(Y_n)_n$ and $(Z_n)_n$ be two sequences of finite metric spaces that are uniformly locally finite and $\uu_n\colon Y_n\to Z_n$ be a sequence of maps.
	\begin{itemize}
		\item The sequence $(\uu_n)_n$ is \emph{$\mathcal{U}$-statistically Lipschitz} if for every $r>0$,
		\[\lim_{C\to \infty}\lim_\mathcal{U}\PP_{Y_n}\left(\left\{y\in Y_n\mid \diam_{Z_n}(\uu_n(B_{Y_n}(y,r))\geq C\right\}\right)= 0.\]

		\item The sequence $(\uu_n)_n$ is \emph{$\mathcal{U}$-statistically expansive} if for every $r>0,$
		\[\lim_{C\to \infty}\lim_\mathcal{U}\PP_{Y_n}\left(\left\{y\in Y_n\mid \diam_{Y_n}(\uu_n^{-1}(B_{Z_n}(\uu_n(y),r))\geq C\right\}\right)= 0.\]

		\item The sequence $(\uu_n)_n$ is \emph{$\mathcal{U}$-statistically a coarse equivalence} if it is $\mathcal{U}$-statistically Lipschitz, $\mathcal{U}$-statistically expansive and $\mathcal{U}$-statistically coarsely bijective.	\end{itemize}
\end{definition}
An example of a $\mathcal{U}$-statistically coarse equivalence is a sequence of quasi-isometries with uniform constants.

\begin{remark}\label{rem:expanImpliesInj}
Note that $\mathcal{U}$-statistically expansivity implies $\mathcal{U}$-statistically coarse injectivity.
\end{remark}

\subsection{Reformulation for sofic approximations of finitely generated groups}

We start by the following useful observation.

\begin{proposition}\label{prop:coarseBij}
Let $G$ and $H$ be two groups each with finite generating set $S_G$ and $S_H$, respectively, and let $(\mathcal{G}_n)_n$ and $(\mathcal{H}_n)_n$ be sequences of sofic approximations of the Cayley graphs $(G,S_G)$ and $(H,S_H)$ respectively.
If $\uu_n: \mathcal{G}_n\to \mathcal{H}_n$ is $\mathcal{U}$-statistically coarsely bijective, then for all $r\geq 0$
\[\lim_\mathcal{U}\PP_{\mathcal{G}_n}\left(\left\{x\in\mathcal{G}_n\mid \uu_n(x)\in \mathcal{H}_n^{(r)}\right\}\right)=1. \]
\end{proposition}
\begin{proof}
Let $\eta= \lim_{\mathcal U}\#\mathcal{H}_n/\#\mathcal{G}_n$, which by
 \cref{prop:compaireSize} is finite. 
 Since balls of bounded radius in $\mathcal G_n$ have bounded size (because $\mathcal G_n$ has bounded degree), $\mathcal{U}$-statistical coarse injectivity implies that 
 \[\lim_{C\to \infty}\lim_\mathcal{U} \PP_{\mathcal{G}_n}\left(\left\{x\in\mathcal{G}_n\mid \#\uu_n^{-1}(x)\leq C\right\}\right)=1. \]
Hence it is enough to prove that for each $C$,
\[\lim_\mathcal{U}\PP_{\mathcal{G}_n}\left(\left\{x\in\mathcal{G}_n\mid \uu_n(x)\notin \mathcal{H}_n^{(r)},\; \#\uu_n^{-1}(x)\leq C\right\}\right)=0
\]
But this quantity is at most 
\[\eta C\lim_\mathcal{U}\PP_{\mathcal{H}_n}\left(\left\{y\in\mathcal{H}_n\mid y\notin \mathcal{H}_n^{(r)}\right\}\right), 
\]
which equals $0$ since $\mathcal{H}_n$ is a sofic approximation.
\end{proof}

Note that a map between two graphs is $C$-Lipschitz if and only if adjacent vertices are sent to vertices at distance at most $C$. Something similar happens  for $\mathcal{U}$-statistically Lipschitz maps between sofic approximations, as shown by the following proposition.

\begin{proposition}\label{prop:StatBiLGroups}
Let $G$ and $H$ be two groups each with finite generating set $S_G$ and $S_H$, respectively, and let $(\mathcal{G}_n)_n$ and $(\mathcal{H}_n)_n$ be sequences of sofic approximations of the Cayley graphs $(G,S_G)$ and $(H,S_H)$ respectively.
Let $\uu_n\colon \mathcal{G}_n\to \mathcal{H}_n$ be a sequence of maps.
The following are equivalent:
	\begin{enumerate}
		\item\label{item:StatBiLGroup1a} for every $s\in S_G$ we have
		\[\lim_{C\to \infty}\lim_\mathcal{U}\PP_{\mathcal{G}_n}\left(\left\{y\in \mathcal{G}_n\mid d_{\mathcal{H}_n}(\uu_n(y),\uu_n(ys))\le C\right\}\right)= 1;\]

		\item\label{item:StatBiLGroup1b} for every $g\in G$ we have
		\[\lim_{C\to \infty}\lim_\mathcal{U}\PP_{\mathcal{G}_n}\left(\left\{y\in \mathcal{G}_n\mid d_{\mathcal{H}_n}(\uu_n(y),\uu_n(yg))\le C\right\}\right)= 1;\]
		\item\label{item:StatBiLGroup1c} the sequence $\uu_n$ is $\mathcal{U}$-statistically Lipschitz
	\end{enumerate}
\end{proposition}
\begin{proof}
Clearly \ref{item:StatBiLGroup1b} implies \ref{item:StatBiLGroup1a}. The fact that \ref{item:StatBiLGroup1c} implies \ref{item:StatBiLGroup1b} results from the following inequality 

\[\PP_{\mathcal{G}_n}\left(\left\{y\in \mathcal{G}_n\mid \diam_{\mathcal{H}_n}(\uu_n(B_{\mathcal{G}_n}(y,|g|))) \le C\right\}\right)\le \PP_{\mathcal{G}_n}\left(\left\{y\in \mathcal{G}_n\mid d_{\mathcal{H}_n}(\uu_n(y),\uu_n(yg))\le C\right\}\right). 
\]

Conversely, we have

\[\PP_{\mathcal{G}_n}\left(\left\{y\in \mathcal{G}_n\mid \diam_{\mathcal{H}_n}(\uu_n(B_{\mathcal{G}_n}(y,R))) \ge C\right\}\right)\le \sum_{|g|\le R}\PP_{\mathcal{G}_n}\left(\left\{y\in \mathcal{G}_n\mid \diam_{\mathcal{H}_n}(\uu_n(B_{\mathcal{G}_n}(y,|g|))) \le C\right\}\right),\]
which gives that \ref{item:StatBiLGroup1b} implies \ref{item:StatBiLGroup1c}.

To complete the proof of the proposition, it thus remains to show that  \ref{item:StatBiLGroup1a} implies \ref{item:StatBiLGroup1b}. This simply follows by writting $g=s_1\ldots s_n$ and arguing by triangle inequality. We leave the details to the reader.
\end{proof}

Below is an analogous result for expansivity.
\begin{proposition}\label{prop:expsofic}
Let $G$ and $H$ be two groups each with finite generating set $S_G$ and $S_H$, respectively, and let $(\mathcal{G}_n)_n$ and $(\mathcal{H}_n)_n$ be sequences of sofic approximations of the Cayley graphs $(G,S_G)$ and $(H,S_H)$ respectively.
Let $\uu_n\colon \mathcal{G}_n\to \mathcal{H}_n$ be a sequence of maps.
\begin{enumerate}
	\item\label{item:StatBiLGroup2} If the sequence $\uu_n$ is $\mathcal{U}$-statistically coarsely surjective, then it is $\mathcal{U}$-statistically expansive if and only if for every $h\in H$ we have
	\[\lim_{C\to \infty}\lim_{\mathcal U}\PP_{\mathcal{G}_n}\left(\left\{x\in \mathcal{G}_n\mid \diam_{\mathcal{G}_n}(\uu_n^{-1}(\uu_n(x))\cup\uu_n^{-1}(\uu_n(x)h))\ge C\right\}\right)= 0.\]
	\item\label{item:StatBiLGroup3} If the sequence $\uu_n$  has $C'$-dense images, and  has preimages  of bounded diameter, then it is $\mathcal{U}$-statistically expansive if and only if for every $h\in B_H(e_H,2C'+1)$, 
	\[\lim_{C\to \infty}\lim_{\mathcal U}\PP_{\mathcal{G}_n}\left(\left\{y\in \uu_n(\mathcal{G}_n)\mid \diam_{\mathcal{G}_n}(\uu_n^{-1}(y)\cup\uu_n^{-1}(yh))\ge C\right\}\right)= 0\]
\end{enumerate}
\end{proposition}

\begin{proof}
Note that $\mathcal{U}$-statistical expansivity obviously implies (\ref{item:StatBiLGroup2}), the useful implication being the converse one.
Assume (\ref{item:StatBiLGroup2}) holds. Then by \cref{prop:coarseBij}, we have
\[\lim_{C\to \infty}\lim_{\mathcal U}\PP_{\mathcal{G}_n}\left(\left\{x\in \mathcal{G}_n\mid \diam_{\mathcal{G}_n}(\uu_n^{-1}(B(\uu_n(x),r))\ge C\right\}\right)\]
is equal to 
\[\lim_{C\to \infty}\lim_{\mathcal U}\PP_{\mathcal{G}_n}\left(\left\{x\in \mathcal{G}_n\mid \diam_{\mathcal{G}_n}(\uu_n^{-1}(B(\uu_n(x),r))\ge C,\; \uu_n(x)\in \mathcal{H}_n^{(r)}\right\}\right).\]
Assuming that $\uu_n(x)\in \mathcal{H}_n^{(r)}$ implies that the ball of radius $r$ around it are vertices of the form $\uu_n(x)h$ for $h\in H$ with $|h|\leq r$.
Hence the latter is at most
\[\sum_{h\in H, |h|\leq r}\lim_{C\to \infty}\lim_{\mathcal U}\PP_{\mathcal{G}_n}\left(\left\{x\in \mathcal{G}_n\mid \diam_{\mathcal{G}_n}(\uu_n^{-1}(\uu_n(x))\cup\uu_n^{-1}(\uu_n(x)h))\ge C\right\}\right)\]
which equals $0$ by (\ref{item:StatBiLGroup2}).

Let us prove the second statement. 
First by our assumption on the diameter of preimages, \ref{item:StatBiLGroup2} is equivalent to having for all $h\in H$,
	\[\lim_{C\to \infty}\lim_{\mathcal U}\PP_{\mathcal{G}_n}\left(\left\{y\in  \uu_n(\mathcal{G}_n)\mid \diam_{\mathcal{G}_n}(\uu_n^{-1}(y)\cup\uu_n^{-1}(yh))\ge C\right\}\right)= 0.\]
Now assume that the latter holds for $|h|\leq 2C'+1$.
We argue by contradiction, letting $h\in H$ be the shortest element for which we do not have convergence. First, note that $|h|\ge 2C'+2$. Thus, we can take $h_1,h_2\in H$ with $h=h_1h_2$, $|h_1|,|h_2|\ge C'+1$ and $|h_1|+|h_2|=|h|$. Next, for every $y\in\mathcal{H}_n$ there exists an $h_y\in B_H(e_H,C')$ such that $zh_1h_y\in \Im(\uu_n)$. Observe that $|h_z^{-1}h_1^{-1}h|=|h_z^{-1}h_2^{-1}|\leq 2C'+1$. 
Hence
\begin{align*}
	&\lim_{C\to \infty}\lim_{\mathcal U}\PP_{\mathcal{G}_n}\left(\left\{y\in \uu_n(\mathcal{G}_n)\mid \diam_{\mathcal{G}_n}(\uu_n^{-1}(y)\cup\uu_n^{-1}(yh))\ge C\right\}\right)\\
	& \le \lim_{C\to \infty}\sum_{h'\in B_H(e_H,C')} \lim_{\mathcal U}\PP_{\mathcal{G}_n}\left(\left\{y\in \uu_n(\mathcal{G}_n)\mid \diam_{\mathcal{G}_n}(\uu_n^{-1}(y)\cup\uu_n^{-1}(yh_1h'))\ge \frac{C}{2}\right\}\right)\\
	& + \lim_{C\to \infty}\sum_{h'\in B_H(e_H,C')} \lim_{\mathcal U}\PP_{\mathcal{G}_n}\left(\left\{y\in \uu_n(\mathcal{G}_n)\mid  \diam_{\mathcal{G}_n}(\uu_n^{-1}(yh_1h')\cup\uu_n^{-1}(yh))\ge \frac{C}{2}\right\}\right).
\end{align*}
Observe that $|h'^{-1}h_1^{-1}h|\leq 2C'+1$, hence both terms equal $0$, leading to a contradiction. This ends the proof.
\end{proof}

\section{Construction of a measure coupling}\label{BoxSpaces}
This section contains the core of the paper, namely the proofs of Theorems \ref{thm:Subgroup}, and \ref{thm:ME}.
It is organized as follows:
First, \S \ref{sec:couplingspace} provides a ``topological coupling'', namely a Polish metric space $\Omega$ on which the two groups admits free commuting actions by homeomorphisms. The topological space in question is roughly the set of functions from $G$ to $H$, slightly modified to ``force'' the $G$-action to be free. This coupling does not require any specific assumption on the groups, but is not yet equipped with an invariant measure. 

In \S \ref{section:non-standard}, we independently construct a measure coupling (subgroup/ME depending on the context), which is obtained as follows. We start taking an ultralimit of probability spaces $(X,\nu)=\lim_\mathcal{U}(\mathcal{G}_n,\PP_{\mathcal{G}_n})$, which comes equipped with a p.m.p. action of $G$, and with a cocycle $\overline{\alpha}: G\times X\to H$, whose existence, relying on the maps $\uu_n$,  is proved in \S \ref{sec:approxCocycle}. This part of the argument has been roughly outlined in the introduction.

We then push forward the measure $\nu$ on the space $\Omega$ through the map $X\to \Omega$ using the cocycle, yielding a measure coupling between the groups (either subgroup or ME depending on the conditions on $\uu_n$).


\subsection{The topological coupling space $\Omega$}\label{sec:couplingspace}

We shall make the following assumptions, depending on whether we aim to construct a ME-coupling between $G$ and $H$ or simply a measure subgroup coupling from $G$ to $H$.
\begin{itemize}
\item In the first case we let $(\mathcal{G}_n)_n$ and $(\mathcal{H}_n)_n$ be sofic approximations of $G$ and $H$ respectively and $\uu_n:\mathcal{G}_n\to \mathcal{H}_n$ be a sequence of maps. Various assumptions will be made on this sequence of maps, but ultimately, we will assume it to be $\mathcal{U}$-statistically a coarse equivalence. 
As we shall see partial results will be obtained under weaker assumptions, in particular assuming that $(\uu_n)_n$ is $\mathcal{U}$-statistically Lipschitz and $\mathcal{U}$-statistically coarsely bijective. 
\item In the second case, we will consider a sequence $\uu_n:\mathcal{G}_n\to H$, where $(\uu_n)_n$ is $\mathcal{U}$-statistically Lipschitz and $\mathcal{U}$-statistically coarsely injective (or simply injective, which should be enough for applications).
\end{itemize}

We start with a informal discussion to motivate our definition of the coupling space below.
The first idea is to work with the space $H^G$, which comes equipped with the action of $G\times H$:
$(g,h)\ast f=hf(g^{-1}\cdot)$. Note that $H$ acts freely, and a fundamental domain $X_H$ for $H$ is simply the set of functions mapping $e_G$ to $e_H$. Moreover the $G$-action is free in restriction to the subspace of injective maps: $G\to H$. 

The second step is to equip the space $H^G$ with a $G\times H$-invariant Borel measure. This can be done exploiting the sequence of maps $(\uu_n)_n$ (see \S \ref{sec:measure}). Under the additional assumption that the $\uu_n$ are  injective, the measure will be supported on injective maps, and so our $G$-action will be essentially free. In this case, $G$ acts essentially freely and it is easy to define a Borel fundamental domain for $G$ as well (in restriction to this subspace), which contains $X_H$. 
But a problem arises if we only assume that $(\uu_n)_n$ are $\mathcal{U}$-statistically coarsely injective. Under this sole assumption, the measure will only be supported on functions with finite pre-images, which is not sufficient to ensure that the $G$-action is essentially free. 

In \cite{Das}, Das considered a special case where $(\uu_n)_n$ were (uniform) quasi-isometries. Since quasi-isometries have pre-images of bounded cardinality, the problem was solved by replacing $H$ by $H\times F$ where $F$ is a finite group, and perturbating the maps $\uu_n$ to make them injective.

Here, in the absence of a uniform bound on the size of preimages, we will have to enlarge the space to $H\times \N$. To force the injectivity of the maps $\uu_n$, we  define auxilliary maps $\rho_n\colon \mathcal{G}_n\to \N$ as follows. 
For every $x\in \mathcal{G}_n$, we fix an arbitrary indexing of $\uu_n^{-1}(\{\uu_n(x)\})=\{x_0,\ldots, x_{m-1}\}$. We then define $\rho_n(x_i)=i$. This provides us with injective maps $\mathcal{G}_n\to \mathcal{H}_n\times \N\colon x\to (\uu_n(x),\rho_n(x))$. Moreover, this leads us to the following definition of coupling space (the relation between the two will become clear in \S \ref{sec:measure}). 

\begin{definition}[Coupling space]
Let $\Omega = (H \times \N)^G = H^G\times\N^G$, endowed with the topology of pointwise convergence, which is a Polish space. We consider the natural (continuous) action of $G\times H$: $(g,h)*(\aaa_H (x),\aaa_\N(x)) = (h\aaa_H(g^{-1}x), \aaa_{\N}(g^{-1}x))$ for $g,x\in G$, $h\in H$ and $(\aaa_H,\aaa_\N)\in (H \times \N)^G$. 
\end{definition}

Note that \[\Omega_0=\left\{(\aaa_H,\aaa_\N)\in \Omega\mid \aaa_H(e_G)=e_H\right\}\] defines a clopen fundamental domain for the action of $H$. 
\begin{definition}[Basis of the topology]
 Let $\Sigma\subseteq G$ be a finite subset, $\aaa\in \Omega$. Then, we define $$W_{\aaa,\Sigma} = \{\cc \in \Omega \mid \forall g\in \Sigma\colon \aaa(g)=\cc(g)\}.$$ 
 We denote by  $\A_\tau$ the collection of subsets  $W_{\aaa,\Sigma}$ and the emptyset.
\end{definition}

Clearly, $\A_\tau$ forms a basis for the topology on $\Omega$. The $\sigma$-algebra generated by $\A_\tau$ contains the entire topology $\tau$ of $\Omega$.  Indeed since $(\Omega,\tau)$ is second countable, every set in $\tau$ is a countable union of basis elements in $\A_\tau$. It follows that $\mathcal{B}(\Omega)=\sigma(\A_\tau)$ and $\left(\Omega,\mathcal{B}(\Omega)\right)$ is standard Borel. Note also that since $G$ is countable, the sets $W_{\aaa,\Sigma}$ are measurable for all (not necessarily finite) subsets $\Sigma\subset G$.

\begin{remark}\label{rem:Cone}
	For all $g\in G$, $h\in H$ and $W_{\aaa,\Sigma}\in \A_\tau$ we have that $$(g,h)*W_{\aaa,\Sigma}=W_{(g,h)*\aaa,g\Sigma}.$$
\end{remark}

\subsection{Almost cocycles}\label{sec:approxCocycle}

Our main tool is a notion of ``almost cocycle" associated to the maps $\uu_n$ and to the ``almost actions" of the groups $G$ and $H$ on their sofic approximations.  

The idea is to associate for every $n\in \N$ and every $x\in \mathcal{G}_n$ a map $T_n(\cdot,x): G\to H$ such that 
$\uu_n(xg)=\uu_n(x)T_n(g,x)$. Such map is defined unambiguously if $\uu_n$ takes values in $H$ (in our ``context of measured subgroups'').

We want to think of $T_n$ as a cocycle, i.e. a map satisfying the following cocycle relation
\begin{equation}\label{eq:Almostcocycle}
T_n(gg',x)=T_n(g,x)T_n(g',xg).
\end{equation}
Of course, such a cocycle does not exist as we don't have actions of $G$ and $H$ on $\mathcal{H}_n^{\mathcal{G}_n}$. Note that  even when such actions exist: e.g.\ if $\mathcal{H}_n$ and $\mathcal{G}_n$ are finite quotients of $H$ and $G$, the fact that $H$-action is not free forces $T_n$ to depend on some arbitrary choice that prevents it from satisfying (\ref{eq:Almostcocycle}). 
Here is a precise definition of $T_n$.

\begin{definition}
Let $G$ and $H$ be two groups with sofic approximations $(\mathcal{G}_n)_n$ and $(\mathcal{H}_n)_n$, and let $\uu_n:\mathcal{G}_n\to \mathcal{H}_n$ (or $\uu_n:\mathcal{G}_n\to H$) be a sequence of maps.
	For any $n\in\N$, we define a map $T_n\colon G\times \mathcal{G}_n\to H\cup\{\infty\}$ as follows. 
	If the following conditions are satisfied  (in case $\uu_n:\mathcal{G}_n\to H$, we omit the second condition): 
	\begin{itemize}
		\item 	$x\in \mathcal{G}_n^{(|g|)}$, 
		\item	$\uu_n(x)\in \mathcal{H}_n^{(r)},$ where $r=d(\uu_n(x),\uu_n(xg))$,	\end{itemize}
		then we let $T_n(g,x)=h$, where $h$ is the unique element of $H$ of length $r$ such that $\uu_n(xg)=\uu_n(x)h$ else we set $T_n(g,x)=\infty$. 
	
\end{definition}

For all $g,g'\in G$, we denote by $\mathcal{C}_n(g,g')$ the set of elements $x\in \mathcal{G}_n^{(|g|)}$ such that (\ref{eq:Almostcocycle}) makes sense and is satisfied. 

\begin{remark}\label{rem:C}
We start observing that $\mathcal{C}_n(g,g')$ is contained in the set $x\in \mathcal{G}_n^{(|g|)}$ such that $\uu_n(x)\in \mathcal{H}_n^{(r_n(x))},$ where $r_n(x)=\max\{d(\uu_n(x),\uu_n(xg)),d(\uu_n(x),\uu_n(xgg')\}$.
\end{remark}
The following proposition shows that $T_n$ plays the role of an ``almost cocycle" in the sense that it will be behave ``statistically'' like a cocycle.

\begin{proposition}\label{prop:approximatecocycle}
	Assuming 
	\begin{itemize}
	\item either that $\uu_n:\mathcal{G}_n\to H$ is {\bf  $\mathcal{U}$-statistically Lipschitz};  	
	\item or that $\uu_n:\mathcal{G}_n\to \mathcal{H}_n$ is {\bf  $\mathcal{U}$-statistically Lipschitz  and  $\mathcal{U}$-statistically coarsely bijective},
\end{itemize}	
then for every $g,g'\in G$,
	\[\lim_\mathcal{U}\PP_{\mathcal{G}_n}\left(\mathcal{C}_n(g,g')\right)= 1,\]
	and  \[\lim_{R\to \infty }\PP_{\mathcal{G}_n}\left( \left\{|T_n(g,x)|_{S_H}\leq R)\right\}\right)=1.\]
\end{proposition}

\begin{proof}
The first case is straightforward: the first statement directly stems from the fact that $\mathcal{G}_n$ is a sofic approximation of $G$, and the second one from the $\mathcal{U}$-statistical Lipschitz condition. Let us focus on the second case, which is more subtle.  By Remark \ref{rem:C} the first statement follows if we can prove that \[\lim_\mathcal{U}\PP_{\mathcal{G}_n}\left(\left\{x\in \mathcal{G}_n^{(r)}\mid \uu_n(x)\in \mathcal{H}_n^{r_n(x)} \right\}\right)= 1,\]
	where $r=\max\{g,gg'\}$, and $r(x)=\max\{d(\uu_n(x),\uu_n(xg)),d(\uu_n(x),\uu_n(xgg')\}$.

Since $(\uu_n)_n$ is statistically Lipschitz, we have 
	\[\lim_{R\to \infty}\lim_\mathcal{U}\PP_{\mathcal{G}_n}\left(\left\{x\in \mathcal{G}_n^{(r)}\mid r_n(x)\leq R \right\}\right)= 1.\]

Hence, it is enough to show that for every $R$, $\lim_\mathcal{U}\PP_{\mathcal{G}_n}\left(\left\{x\in \mathcal{G}_n^{(r)}\mid \uu_n(x)\notin \mathcal{H}_n^{(R)} \right\}\right)=0$, or simply that 
	\[\lim_\mathcal{U}\PP_{\mathcal{G}_n}\left(\left\{x\in \mathcal{G}_n\mid \uu_n(x)\notin \mathcal{H}_n^{(R)} \right\}\right)= 0.\]
 Now, since $\uu_n$ is  $\mathcal{U}$-statistically coarsely injective, we have that $\# \uu_n^{-1}(\uu_n(x))$ is bounded up to an arbitrarily small proportion of $x\in \mathcal{G}_n$. Therefore we are left to show that 
	\[\lim_\mathcal{U}\frac{\#\left(\mathcal{H}_n^{(R)}\right)^c}{\# \mathcal{G}_n}= 0,\]
	Finally, since the sequence $(\uu_n)_n$ is  $\mathcal{U}$-statistically coarsely surjective, we have, by \cref{prop:compaireSize}, that $\# \mathcal{G}_n\le C\# \mathcal{H}_n$ for some constant $C>0$, so the result follows from the fact that $\mathcal{H}_n$ is a sofic approximation. Thus, the first statement follows. The second statement results from the fact that for every $g$, with probability tending to $1$, $|T_n(g,x)|_{S_H}=d_{\mathcal{H}_n}(\uu_n(xg),\uu_n(x))$, combined with the fact that $\uu_n$ is $\mathcal{U}$-statistically Lipschitz (we leave the straightforward details to the reader). 

\end{proof}		

We end this paragraph with the following facts that will come in complement of the previous proposition. Recall that $\rho_n$ is the auxilliary function  $\mathcal{G}_n\to \N$ defined in \S \ref{sec:couplingspace}.

\begin{proposition}\label{prop:rho}
	Assuming that the sequence $(\uu_n)$ is {\bf $\mathcal{U}$-statistically coarsely injective}. Then, 
	\[\lim_{N\to \infty}\lim_\mathcal{U}\PP_{\mathcal{G}_n}\left(\left\{x\in \mathcal{G}_n\mid \rho_n(x)\leq N\right\}\right)= 1.\]
Moreover, for every $g\in G\setminus\{e_G\}$,
	\[\lim_\mathcal{U}\PP_{\mathcal{G}_n}\left(\left\{x\in \mathcal{G}_n\mid \uu_n(xg)=\uu_n(x) \;{\text and}\; \rho_n(xg)=\rho_n(x)\right\}\right)= 0,\]
	 and
	\[\lim_\mathcal{U}\PP_{\mathcal{G}_n}\left(\left\{x\in \mathcal{G}_n\mid \exists g\in G,\; \uu_n(x)=\uu_n(xg) \textnormal{ and }\rho_n(xg)= 0\right\}\right)= 1.\]
	In other words, 
	\[\lim_\mathcal{U}\PP_{\mathcal{G}_n}\left(\left\{x\in \mathcal{G}_n\mid \exists g\in G,\; T_n(g,x)=e_H \textnormal{ and }\rho_n(xg)= 0\right\}\right)= 1.\]
\end{proposition}
\begin{proof}
By definition of $\rho_n$, $\rho_n(x)$ is at most the cardinality of $\uu_n^{-1}(\uu_n(x))$. Hence the first statement is an immediate consequence of $\mathcal{U}$-statistical coarse injectivity. Let $r_n:=\diam (\uu_n^{-1}(\uu_n(x)))$.
The second statement follows from the definition of $\rho_n$ and the fact that  $\lim_\mathcal{U}\PP_{\mathcal{G}_n}\left(\mathcal{G}_n^{(|g|)}\right)=1$. Indeed, if $x\in \mathcal{G}_n^{(|g|)}$, then $xg$ exists and is distinct from $x$, which means that either $\uu_n(x)\neq \uu_n(xg)$, or that  $\rho_n(xg)\neq \rho_n(x)$.

	For the last statement, note that if $x\in \mathcal{G}_n^{(C)}$ and $\diam(\uu_n^{-1}(\uu_n(x)))\le C$, then there exists $g\in \uu_n^{-1}(\uu_n(x))$ of length $\leq C$ such that $\rho_n(xg)= 0$. Hence, the conclusion follows from the facts that $\lim_\mathcal{U}\PP_{\mathcal{G}_n}\left(\mathcal{G}_n^{(C)}\right)=1$, and the $\mathcal{U}$-statistical coarse injectivity of $\uu_n$.
\end{proof}

\subsection{A non-standard coupling measure space $(\overline{\Omega},\overline{\mu})$}\label{section:non-standard}
In this section, we assume that $\uu_n:\mathcal{G}_n\to \mathcal{H}_n$ satisfy the conditions of Proposition \ref{prop:approximatecocycle}, i.e.\ 
\begin{itemize}
	\item either that $\uu_n:\mathcal{G}_n\to H$ and  $\uu_n:\mathcal{G}_n\to \mathcal{H}_n$ is {\bf  $\mathcal{U}$-statistically Lipschitz};	\item or that $\uu_n:\mathcal{G}_n\to \mathcal{H}_n$ is {\bf  $\mathcal{U}$-statistically Lipschitz  and  $\mathcal{U}$-statistically coarsely bijective},
\end{itemize}	

Here, we construct a non-standard measure coupling between $G$ and $H$. 

We define \[(X,\nu)=[\mathcal{G}_n]_{\mathcal{U}},\]
as in \S \ref{sec:Loeb}.
The fact that $\mathcal{G}_n$ is a sofic approximation implies that for every $g\in G$, $\PP_{\mathcal{G}_n}(\mathcal{G}_n^g)\to 1$ as $n\to \infty$. In other word, the subset 
\[X'=\bigcap_{g\in G}[\mathcal{G}_n^g]_\mathcal{U}\]
has full measure in $X$. We denote by $\nu'$ the restriction of $\nu$ to $X'$.
We deduce that $G$ acts by measure-preserving transformation on $(X',\nu')$ through the following formula
\[g\cdot [x_n]_\mathcal{U}=[x_ng^{-1}]_\mathcal{U}.\]
One easily checks that the corresponding action is free.

By Proposition \ref{prop:approximatecocycle}, $[\mathcal{C}_n(g,g')]_\mathcal{U}$ has full measure in $X$, as well as $\bigcup_R\left\{x\mid |T_n(g,x)|_{S_H}\leq R)\right\}$, for all $g,g'\in G$ (observe that both are contained in $X'$). Therefore, for all $x\in Y$, $T(g,x)=[T_n(g,x_n)]_\mathcal{U}$ is a well-defined map from $G$ to $H$.
Also note that, by compactness of $\N\cup \{\infty\}$, $\rho(x):=[\rho_n(x_n)]_\mathcal{U}$ exists, and by Proposition \ref{prop:rho}, there exists a subset $Y\subset X$  of full measure such that for all $x\in Y$, $\rho(x)\in \N$, $g\to (T(g,x),\rho(g^{-1}\cdot x))$ is injective and there exists $g\in G$ such that $T(g,x)=e_H$ and $\rho(g^{-1}\cdot x)=0$.

Let \[X''=\bigcap_{g,g'\in G} [\mathcal{C}_n(g,g')]_\mathcal{U},\]
\[X'''= \bigcap_{g\in G}\bigcup_R [\left\{x\mid |T_n(g,x)|_{S_H}\leq R)\right\}]_\mathcal{U},\] 
and let 
\[\overline{\Omega}_0 = \bigcap_{g\in G}g\cdot (X''\cap X'''\cap  Y).\] 
Note that $\overline{\Omega}_0$ is a full-measure, $G$-invariant subset of $X'$. We denote by $\bar{\mu}_0$ the restriction of $\nu$ to $\overline{\Omega}_0$.

Note that by definition of $\overline{\Omega}_0$ the restriction of $T$ to 
$G\times \overline{\Omega}_0$ satisfies the cocycle relation
\begin{equation}\label{eq:cocycle}
T(gg',x)=T(g,x)T(g',g^{-1}\cdot x).
\end{equation}

We now consider the space $(\overline{\Omega},\overline{\mu})=(\overline{\Omega}_0,\overline{\mu}_0)\times H$, where $H$ is equipped with the counting measure. 
We define a measure preserving action of $G\times H$-action on $(\overline{\Omega},\overline{\mu})$ in the usual way:
\[(g,h)\ast (x,k)=(g\cdot x,hkT(g^{-1},x)).\]
Observe that since the action of $G$ on $\overline{\Omega}_0$ is free, we deduce that the $G\times H$ action on $\overline{\Omega}$ is free.

Denote $\overline{\alpha}:\overline{\Omega}_0\times G\to H$, the corresponding cocycle, defined such that $\overline{\alpha}(g,x)$ is the unique element $h\in H$ such that $(g,h)\ast x\in \overline{\Omega}_0$. We deduce that $\overline{\alpha}(g,x)=T(g^{-1},x)^{-1}$.

We gather the main conclusions of this discussion in the following proposition.

\begin{proposition}\label{prop:nonstandardCoupling}
The measure space $(\overline{\Omega},\overline{\mu})$ comes equipped with a free measure-preserving action of $G\times H$, such that $\overline{\Omega}_0$ is a fundamental domain  of $H$ with measure $1$. The associated cocycle $\overline{\alpha}:\overline{\Omega}_0\times G\to H$, satisfies $\overline{\alpha}(g,x)=T(g^{-1},x)^{-1}$. Besides, for all $x\in \overline{\Omega}_0$, $g\to (T(g,x),\rho(g^{-1}\cdot x))$ is injective. Finally, for all $x\in \overline{\Omega}_0$, there exists $g\in G$ such that $T(g,x)=e_H$ and $\rho(g^{-1}\cdot x)=0$.
\end{proposition}

Let us conclude this section with a remark about the map $\rho$. Clearly the coupling space $\overline{\Omega}$ does not at all involve $\rho$. But the reason why we keep track of it is that it will be needed in the next section when factorizing $\overline{\Omega}$ to our topological coupling space $\Omega$ from \S \ref{sec:couplingspace}.

\subsection{From $\overline{\Omega}$ to $\Omega$}\label{sec:measure}
Here, we will factor our non-standard coupling constructed in the previous section to the topological coupling space $\Omega=(H \times \N)^G$ described in \S \ref{sec:couplingspace}. 

\subsubsection{Factorising to $\Omega$}
We now define a $G\times H$-equivariant map $J:\overline{\Omega}\to \Omega$ as follows: 
We first define the following map $J_0:\overline{\Omega}_0\to \Omega_0$ by
\[J_0(x)(g)=(T(g,x),\rho(g^{-1}\cdot x))\]
We finally define $J:\overline{\Omega}=\overline{\Omega}_0\times H\to \Omega$ as follows

 \[J(x,h)(g)=(e_G,h)\ast J_0(x)(g)=(hT(g,x),\rho(g^{-1}\cdot x)).\]
Equivariance with respect to the $H$-action is obvious,  so let us check the $G$-equivariance: pick $g,g'\in G$, and let us calculate
\[J((g,e_H)\ast (x,h))(g')=J(g\cdot x,hT(g^{-1},x))(g')=(hT(g^{-1},x)T(g',g\cdot x),\rho(g'^{-1}\cdot (g\cdot x)))\]
Using (\ref{eq:cocycle}), we have $T(g^{-1},x)T(g',g\cdot x)=T(g^{-1}g',x)$, while $\rho(g'^{-1}\cdot (g\cdot x))=\rho((g^{-1}g')^{-1}\cdot x)$.
Hence 
\[J((g,e_H)\ast (x,h))(g')=(hT(g^{-1}g',x),\rho((g^{-1}g')^{-1}\cdot x))=J(x,h)(g^{-1}g'),\]
and we are done.

\begin{lemma}
The map $J$ is measurable with respect to the $\sigma$-algebras $\mathcal{B}(\Omega)$ and $\mathcal{B}_0(\overline{\Omega})$, where the latter is the $H$-invariant $\sigma$-algebra on $\overline{\Omega}$ spanned by subsets of the form  
$[A_n]_\mathcal{U}$, for any sequence $A_n\in  \mathcal{G}_n$.
\end{lemma}
\begin{proof}
To check measurability, it is enough to show that $J^{-1}({W_{\aaa,\Sigma}})\in \mathcal{B}_0(\overline{\Omega})$ for every $W_{\aaa,\Sigma}\in \B_\tau$. On translating by a suitable element of $G\times H$, we are reduced to the case where $e_G\in \Sigma$ and $\aaa(e_G)=1$, in which case $W_{\aaa,\Sigma}\in \Omega_0$. It is then enough to check that $J_0^{-1}(W_{\aaa,\Sigma})\in \mathcal{B}_0(\overline{\Omega_0})$. But it follows from its definition that $x=[x_n]_\mathcal{U}\in J_0^{-1}(W_{\aaa,\Sigma})$ if and only if for all $b\in \Sigma$, $T(x,b)=\aaa_H(b)$ and $\rho(b\cdot x)= \aaa_\N(b)$, which translates into: for all $b\in \Sigma$, and for $\mathcal{U}$-almost all $n$, $T_n(x_n,b)=\aaa_H(b)$ and $\rho_n(b\cdot x_n)= \aaa_\N(b)$. Hence,
\begin{multline*}
J^{-1}(W_{\aaa,\Sigma})=J_0^{-1}(W_{\aaa,\Sigma}) \times \{e_H\} =   \left\{x\in X \mid T(b,x)=\aaa_H(b), \rho(x_nb)= \aaa_\N(b)\right\} \times \{e_H\} \\
= \left[\left\{x_n\in \mathcal{G}_n\mid T_n(b,x_n)=\aaa_H(b), \rho_n(x_nb) = \aaa_\N(b)\right\}\right]_\mathcal{U} \times \{e_H\},
\end{multline*}                                                                                    
which belongs to  $\mathcal{B}_0(\overline{\Omega})$, so we are done.
\end{proof}

We can therefore push the measure $\overline{\mu}$ to a $G\times H$-invariant Borel measure $\mu$ on $\Omega$. We claim that $J_0(\overline{\Omega}_0)$ is contained in the Borel subset $\Omega''_0$ of $\Omega_0$ consisting of injective maps $\aaa=(\aaa_H,\aaa_\N)$ such that $\aaa_H(e_G)=e_H$, and such that for all $g\in G$, there exists $g'\in G$ such that $\aaa(g')=(\aaa_H(g),0)$. Indeed, let $\aaa$ be in $J_0(\overline{\Omega}_0)$. Hence there is $x\in \overline{\Omega}_0$ such that $\aaa(g)=(T(g,x),\rho(g^{-1}\cdot x))$. Injectivity of $\aaa$ follows from \cref{prop:nonstandardCoupling}.  Fix some $g\in G$, and observe that $J_0(g^{-1}\cdot x)=(\aaa_H(g\; \cdot) \aaa_H(g)^{-1}, \aaa_\N(g\;\cdot)$. Since $\overline{\Omega}_0$ is $G$-invariant, we have $g\cdot x_0\in \overline{\Omega}_0$. By \cref{prop:nonstandardCoupling}, there exists $g'$ such that $T(g',g^{-1}\cdot x)=e_H$ and $\rho(g'^{-1}g^{-1}\cdot x)=0$. This translates into $\aaa_H(gg')=\aaa(g)$ and $\aaa_\N(gg'\;\cdot)=0$, and our claim is proved.

Likewise,  $J(\overline{\Omega})$ is contained in the Borel subset $\Omega''$ of $\Omega$ consisting of injective maps $\aaa$ such that for all $g\in G$, there exists $g'\in G$ such that $\aaa(g')=(\aaa_H(g),0)$. 
This yields the following proposition.
 
\begin{proposition}\label{prop:existenceofmu}
	There exists a Borel measure $\mu$ on $\Omega$ that is $G\times H$-invariant and such that $\mu(\Omega_0)=1$. Moreover, $\mu$ is supported on the Borel subset consisting of injective maps $\aaa$ such that for all $g\in G$, there exists $g'\in G$ such that $\aaa(g')=(\aaa_H(g),0)$. 
\end{proposition}

\subsubsection{Computing the measure of basic subsets}\label{sec:BasicSubsets}

We close this section by computing the measure of subsets in $\A_\tau$ in terms of ultralimits of measures of subsets of $\mathcal{G}_n$. Similar computations will be crucial in the next paragraph for proving that $G$ admits a fundamental domain of finite measure. We point out that such computations can be ignored for statements relative to measure subgroups.

Note that for any $W=W_{\aaa,\Sigma}$, with $\Sigma\subset G$ finite and non-empty, we have 
\[J^{-1}(W)=\Big\{(x,h)\in \overline{\Omega}\mid
\forall g\in \Sigma\colon T(g,x)=h\aaa_H(g), \;\rho(g^{-1}\cdot x)= \aaa_\N(g)\Big\}.\]

Since $\Sigma$ is non-empty, for each $x\in \overline{\Omega}_0$, there is at most one $h\in H$ such that $(x,h)\in J^{-1}(W)$.
Hence we deduce that $J^{-1}(W)$ has the same measure as
\begin{equation}\label{eq:ProjectedSet}
\Big\{x\in \overline{\Omega}_0\mid \exists h\in H\colon \forall g\in \Sigma\colon T(g,x)=h\aaa_H(g), \;\rho(g^{-1}\cdot x)= \aaa_\N(g)\Big\}.
\end{equation}
Fix some $g_0\in \Sigma$.  Then for any $x$ in the above set we have that $h$ must necessarily be equal to $T(g_0,x)\aaa_H(g_0)^{-1} = \lim_\mathcal{U}T_n(g_0,x_n)\aaa_H(g_0)^{-1}$. Hence, the set in equation (\ref{eq:ProjectedSet}) is equal to\footnote{Note that if $\Sigma=\{g_0\}$, then this is simply $\{x\in \overline{\Omega}_0\mid \rho(g_0^{-1}\cdot x)= \aaa_\N(g_0)\}.$}
\[\Big\{x\in \overline{\Omega}_0\mid  \forall g\in \Sigma\colon T(g,x)=T(g_0,x)\aaa_H(g_0)^{-1} \aaa_H(g), \;\rho(g^{-1}\cdot x)= \aaa_\N(g)\Big\}.\]

In practice, we will often use it when $g_0=e_G$, in which case $T(e_G,x)=e_H$, and so the above set is
 \[\Big\{x\in \overline{\Omega}_0\mid  \forall g\in \Sigma\colon T(g,x)=\aaa_H(e_G)^{-1} \aaa_H(g), \;\rho(g^{-1}\cdot x)= \aaa_\N(g)\Big\}.\]
Let us briefly pause here to interpret this formula. Let us ignore the expression involving $\rho$ (which is a technical adjustment). Here what we see is $ T(g,x)=\aaa_H(e_G)^{-1} \aaa_H(g)$, which means that the map $T(\cdot, x):G\to H$ has to coincide on $\Sigma$ with $\aaa$, conveniently translated so that it maps $e_G$ to $e_H$.
 Now using that $T(g,x)=[T_n(g,x_n)]_\mathcal{U}$, we have
\[\mu(W)=\lim_\mathcal{U}\PP_{\mathcal{G}_n}\Big(\Big\{x\in \mathcal{G}_n\mid \forall g\in \Sigma\colon T_n(g,x)=\aaa_H(e_G)^{-1} \aaa_H(g), \;\rho_n(xg)= \aaa_\N(g)\Big\}\Big).\]
Finally using the definition of $T_n$, this becomes
\[\mu(W)=\lim_\mathcal{U}\PP_{\mathcal{G}_n}\Big(\Big\{x\in \mathcal{G}_n\mid \forall g\in \Sigma\colon \uu_n(xg)=\uu_n(x)\aaa_H(e_G)^{-1} \aaa_H(g), \;\rho_n(xg)= \aaa_\N(g)\Big\}\Big).\]

\section{A fundamental domain for the $G$ action on $(\Omega,\mu)$}\label{sec:FundDomG}
The construction of \S \ref{BoxSpaces} only relied on the assumption  that the sequence $\uu_n:\mathcal{G}_n\to \mathcal{H}_n$ is {\bf $\mathcal{U}$-statistically Lipschitz}  and {\bf $\mathcal{U}$-statistically coarsely bijective}, or alternatively that $\uu_n:\mathcal{G}_n\to H$ is {\bf $\mathcal{U}$-statistically Lipschitz}. These assumptions allowed us to construct a coupling space $(\Omega,\mu)$ such that the $H$ action admits a (Borel) fundamental domain $X_H=\Omega_0$ such that  $\mu(X_H)=1$. 
We now turn our attention to the action of $G$. 
 By \cref{prop:existenceofmu},  the measure $\mu$ is supported in the Borel subspace $\Omega''$ consisting of injective maps $\aaa: G\to H\times \N$ such that for all $g\in G$, there exists $g'\in G$ such that $\aaa(g')=(\aaa_H(g),0)$. A quick way to construct a Borel fundamental domain for the action of $G$ on $\Omega''$ goes as follows: take a well-order $<$ on $H\times \N$ and let the fundamental domain be the subset of $\Omega''$ consisting of (injective) maps $G\to H\times \N$ which attain their $<$-minimum at $e_G$. This implies at least that we automatically get a measure subgroup coupling for both constructions.

\begin{proposition}\label{prop:SubgroupCoupling}
Assume  that $\uu_n:\mathcal{G}_n\to H$ are $\mathcal{U}$-statistically Lipschitz  and $\mathcal{U}$-statistically coarsely injective, or $\uu_n:\mathcal{G}_n\to \mathcal{H}_n$ are $\mathcal{U}$-statistically Lipschitz  and $\mathcal{U}$-statistically coarsely bijective. Then $(\Omega'',\mu)$ is a subgroup coupling from $G$ to $H$, where $X_H=\Omega''_0$, and whose associated cocycle $\alpha$ satisfies $\alpha(g,\aaa)=T(g^{-1},\aaa_H)^{-1}=\aaa_H(g^{-1})^{-1}$, for all $\aaa=(\aaa_H,\aaa_\N)\in \Omega''_0$.
\end{proposition}

In order to obtain a measure equivalence coupling, these assumptions won't be enough. We shall need to assume in addition that $\uu_n:\mathcal{G}_n\to \mathcal{H}_n$ is $\mathcal{U}$-statistically expansive (hence $\mathcal{U}$-statistically a coarse equivalence). The main goal of this subsection and the following one will be to prove the following proposition.
\begin{proposition}\label{prop:OECoupling}
Under the assumption  that $\uu_n:\mathcal{G}_n\to \mathcal{H}_n$ are $\mathcal{U}$-statistically a coarse equivalence, then  $(\Omega'',\mu)$ defines an ME-coupling from $G$ to $H$. Besides, if $(\uu_n)$ is $\mathcal{U}$-statistically coarsely bijective, $(\Omega'',\mu)$ defines an OE coupling from $G$ to $H$, where $X_H=\Omega''_0=X_G$.
\end{proposition}

 In order to estimate the volume of $\Omega''/G$, it will be convenient to deal with a different fundamental domain, which can be defined as follows.

We define \[Y_0= \{\aaa\in \Omega''\mid \aaa(e_G)=(e_H,0)\}.\]
Enumerate the elements of $H=\{h_0=e_H,h_1,h_2,\ldots\}$ and define recursively for $i\geq 1$
\[Y_{i}=h_{i}\ast Y_0\setminus G\ast\bigcup_{j=0}^{i-1}Y_j.\]
Finally define $X_G=\bigsqcup_{i}Y_i$. 

In words, $Y_i$ is the set of functions $\aaa\in \Omega''$ such that $\aaa_H$ does not take values $h_j$, for any $j<i$, and $\aaa$ takes value $(h_i,0)$ exactly at $e_G$. Clearly, the $Y_i$ are disjoint. Moreover $G\ast Y_i$ (union of $G$-translates of $Y_i$) consists of functions $\aaa$ such that $\aaa_H$ does not take values $h_i$ for $i<j$, but takes value $h_j$: indeed, recall that functions of $\Omega''$ have the property that if $\aaa_H$ takes the value $h$, then $\aaa$ takes the value $(h,0)$. We deduce that $G\ast X_G=\bigcup_i G\ast Y_i$ is all of $\Omega''$.
Clearly,  $G\ast Y_i$ and $G\ast Y_j$ are disjoint if $i\neq j$. Also, since functions of  $\Omega''$ are injective, $G$-translates of $Y_i$ are disjoint. Hence $G$-translates are disjoint. We have therefore proved the following proposition.
 
\begin{proposition}\label{prop:X_H}
	$X_G$ is a Borel fundamental domain for the action of $G$ on $\Omega''$.
\end{proposition} 

\begin{remark}\label{rem:Yi}
It follows from the previous discussion that 
\begin{eqnarray*}
Y_i& = & \{\aaa\in \Omega'' \mid  \aaa(e_G)=(h_i,0),  \;{\text and }\;\forall j<i, \;\forall g\in G,\; \aaa_H(g)\neq h_j \}\\
     & = &  \bigcap_{r\geq 0}\{\aaa\in \Omega'' \mid  \aaa(e_G)=(h_i,0),  \;{\text and }\;\forall j<i,\; \forall g\in B_G(e_G,r)  \; \aaa_H(g)\neq h_j \}.
\end{eqnarray*}
Reasoning as in \S \ref{sec:BasicSubsets}, we get
\[\mu(Y_i) = \lim_{r\to\infty}\lim_{\mathcal{U}} \PP_{\mathcal{G}_n}\Big(\Big\{x\in\mathcal{G}_n\mid \rho_n(x)=0, 
 \forall j<i\colon \uu_n(x)h_i^{-1}h_j\notin \uu_n(B_{\mathcal{G}_n}(x,r)) \Big\}\Big).\]
\end{remark}

We are now ready to estimate the covolume of the $G$-action. 
\begin{proposition}\label{prop:fundDom}
	Assuming that  $\uu_n:\mathcal{G}_n\to \mathcal{H}_n$ are $\mathcal{U}$-statistically Lipschitz  and $\mathcal{U}$-statistically coarsely bijective, we have $\mu(X_G)\ge\lim\limits_{\mathcal{U}}\frac{\#\mathcal{H}_n}{\#\mathcal{G}_n}$. If we assume, in addition, that $\uu_n$ are $\mathcal{U}$-statistically expansive (i.e.\  is $\mathcal{U}$-statistically a coarse equivalence), then the latter inequality becomes an equality.
\end{proposition}
\begin{proof}
	Recall that $H=\{e_H=h_0,h_1,\ldots\}$, and assume without loss of generality that $|h_i|\le|h_{i+1}|$ for every $i$. Thus, by our computation of $\mu(Y_i)$ and the fact that they are pairwise disjoint, we have
	\[\mu(X_G) 
	 = \sum_{i=0}^\infty \mu(Y_i)= \lim_{R\to \infty}  \sum_{h_i\in B(e_H,R)} \mu(Y_i).\]
Hence $\mu(X_G)$  is equal to
	
	\begin{align*}
	& \lim_{R\to \infty} \lim_{r\to\infty}\lim_{\mathcal{U}}\sum_{h_i\in B_H(e_H,R)} \PP_{\mathcal{G}_n}\Big(\Big\{x\in\mathcal{G}_n\mid \rho_n(x)=0, 
 \forall j<i\colon \uu_n(x)h_i^{-1}h_j\notin \uu_n(B_{\mathcal{G}_n}(x,r)) \Big\}\Big)\\
	& \ge \lim_{R\to \infty}\lim_{\mathcal{U}}\sum_{h_i\in B_H(e_H,R)} \PP_{\mathcal{G}_n}\Big(\Big\{x\in\mathcal{G}_n\mid \rho_n(x)=0, \forall j<i\colon \uu_n(x)h_i^{-1}\notin \Im(\uu_n)h_j^{-1} \Big\}\Big).
	\end{align*}

	Recall that for every $y\in \Im(\uu_n)$ there exists a unique $x\in \mathcal{G}_n$ such that $\rho_n(x)=0$ and $\uu_n(x)=y$. So,
	\begin{align*}
	\mu(X_G) & \ge \lim_{R\to \infty} \lim_{\mathcal{U}}\sum_{h_i\in B_H(e_H,R)} \frac{\#(\Im(\uu_n)h_i^{-1}\setminus \bigcup_{j<i}\Im(\uu_n)h_j^{-1})}{\#\mathcal{G}_n}
	\\
	& = \lim_{R\to \infty} \lim_{\mathcal{U}} \frac{1}{\#\mathcal{G}_n}\#\big\{y\in\mathcal{H}_n\mid d_{\mathcal{H}_n}(y,\Im(\uu_n))\le R \big\}\\
	& =  \lim_{\mathcal{U}} \frac{\#\mathcal{H}_n}{\#\mathcal{G}_n},
	\end{align*}
	as the maps $\uu_n$ are $\mathcal{U}$-statistically coarsely surjective.

	We now turn to the upper bound:
	
	\begin{align*}
	\mu(X_G) 	
	& = \lim_{R\to \infty} \lim_{r\to\infty}\lim_{\mathcal{U}}\sum_{h_i\in B_H(e_H,R)} \PP_{\mathcal{G}_n}\Big(\Big\{x\in\mathcal{G}_n\mid \rho_n(x)=0, \\
	& \AlignRight{ \forall j<i\colon \uu_n(x)h_i^{-1}h_j\notin \uu_n(B_{\mathcal{G}_n}(x,r)) \Big\}\Big)}.\\
	& \le \lim_{R\to \infty} \lim_{r\to\infty}\lim_{\mathcal{U}}\sum_{h_i\in B_H(e_H,R)} \PP_{\mathcal{G}_n}\Big(\Big\{x\in\mathcal{G}_n\mid \rho_n(x)=0, \forall j<i\colon \uu_n(x)h_i^{-1}\notin \Im(\uu_n)h_j^{-1}  \Big\}\Big)\\
	& + \lim_{R\to \infty} \lim_{r\to\infty}\lim_{\mathcal{U}}\sum_{h_i\in B_H(e_H,R)}  \PP_{\mathcal{G}_n}\Big(\Big\{x\in\mathcal{G}_n \mid \exists x'\colon \; d_{\mathcal{G}_n}(x,x')>r,\\
	& \AlignRight{\exists j<i\colon \uu_n(x)h_i^{-1}= \uu_n(x')h_j^{-1} \Big\}\Big) }\\
	& = \lim_{\mathcal{U}} \frac{\#\mathcal{H}_n}{\#\mathcal{G}_n} + \lim_{R\to \infty} \sum_{h_i\in B_H(e_H,R)}\lim_{r\to\infty}\lim_{\mathcal{U}} \PP_{\mathcal{G}_n}\Big(\Big\{x\in\mathcal{G}_n \mid \exists x'\colon \; d_{\mathcal{G}_n}(x,x')>r,\qquad\\
	& \AlignRight{d_{\mathcal{H}_n}(\uu_n(x),\uu_n(x')) \le 2R \Big\}\Big) }\\
	& = \lim_{\mathcal{U}} \frac{\#\mathcal{H}_n}{\#\mathcal{G}_n}.
	\end{align*}
Let us explain:  for the first equality, we considered two cases, depending on whether $\uu_n(x)h_i^{-1}h_j$ is or not in $\Im(\uu_n)$. When it is  (the second term), then it equals $\uu_n(x')$ for some $x'$ such that $d_{\mathcal{G}_n}(x,x')>r$, which explains the inequality. The second equality follows from the computation we had performed just above. Finally we conclude thanks to the fact that the maps $\uu_n$ are $\mathcal{U}$-statistically expansive.

Thus, $\mu(X_G)=\lim\limits_{\mathcal{U}} \frac{\#\mathcal{H}_n}{\#\mathcal{G}_n}$ and we are done.
\end{proof}
Collecting the above results, we deduce the following theorem.
\begin{proposition}\label{prop:ME}
	Let $G$ and $H$ be two finitely generated groups with sofic approximations $(\mathcal{G}_n)_n$ and $(\mathcal{H}_n)_n$. Assume that  the  sequence of maps $\uu_n\colon \mathcal{G}_n\to \mathcal{H}_n$ is $\mathcal{U}$-statistically a coarse equivalence. Then, there exists a measure equivalence coupling $(\Omega,\mu)$ between be the groups $G$ and $H$ such that 
	\[\frac{\mu(\Omega/G)}{\mu(\Omega/H)}=\lim_{\mathcal{U}}\frac{\#\mathcal{H}_n}{\#\mathcal{G}_n}.\]	.
\end{proposition}
\begin{proof}
	This is an immediate consequence of Propositions \ref{prop:existenceofmu}, \ref{prop:X_H} and \ref{prop:fundDom}.
\end{proof}

\subsection{A criterion implying $X_G=X_H$ a.e.}

Proposition \ref{prop:OECoupling} results from the following more general fact.

\begin{proposition}\label{prop:fundamentalDomainBis}
	Assume that $\uu_n:\mathcal{G}_n\to \mathcal{H}_n$ are $\mathcal{U}$-statistically Lipschitz  and $\mathcal{U}$-statistically coarsely bijective.
	\begin{itemize}
		\item If in addition the maps $\uu_n$ have fibers with diameter at most $R$, then $\mu(X_H\setminus B_G(e_G,R)*X_G)=0$. In particular, if the maps are  injective, then $\mu(X_H\setminus X_G)=0$.
		\item 
		If the maps $\uu_n$ are $\mathcal{U}$-statistically a coarse equivalence and have $C$-dense image, then $\mu(X_G\setminus B_H(e_H,C)*X_H)=0$. In particular, if the maps are surjective, then $\mu(X_G\setminus X_H)=0$. 
		\item if both conditions are satisfied, then the resulting coupling is mutually cobounded, and if $\uu$ is  bijective, then $X_H=X_G$ a.e.
	\end{itemize}
\end{proposition}
\begin{proof}
The third item is simply a consequence of the first two. 
	So suppose that the diameters of the fibers of the maps $\uu_n$ are at most $R$. Then, we get
	\begin{align*}
	\mu(X_H \! \setminus\! B_G(e_G,R)*X_G) & \le \mu(X_H \! \setminus\! B_G(e_G,R)*Y_0) \\
	& = \lim\limits_{\mathcal{U}} \PP_{\mathcal{G}_n}\Big(\Big\{x\in \mathcal{G}_n\mid \forall g\!\in\! B_G(e_G,R)\colon \uu_n(xg)\neq \uu_n(x) \text{ or } \rho_n(xg)\neq 0 \Big\}\Big)\\
	& = \lim\limits_{\mathcal{U}} \PP_{\mathcal{G}_n}\Big(\Big\{x\in \mathcal{G}_n\mid \forall x'\in \uu_n^{-1}(\uu_n(x))\colon  \rho_n(x')\neq 0 \Big\}\Big)\\
	& = 0,
	\end{align*}
	as every fiber contains an element $x'$ such that $\rho_n(x')=0$. 
		
	To show the second item, it suffices to show that $\mu(Y_i)=0$ whenever $|h_i|>C$. We denote $\eta= \lim_{\mathcal{U}} \frac{\#\mathcal{H}_n}{\#\mathcal{G}_n}$.
	Recall that by Remark \ref{rem:Yi}, we have 
	\begin{align*}
	\mu(Y_i) & = \lim_{R\to\infty}\lim_{\mathcal{U}} \PP_{\mathcal{G}_n}\Big(\Big\{x\in\mathcal{G}_n\mid \rho_n(x)=0, \forall j<i\colon \uu_n(x)h_i^{-1}\notin \uu_n(B_{\mathcal{G}_n}(x,R))h_j^{-1} \Big\}\Big)
	\\
	& \le \lim_{R\to\infty}\lim_{\mathcal{U}}\PP_{\mathcal{G}_n}\Big(\Big\{x\in\mathcal{G}_n\mid\forall g\in B_G(e_G,R)\colon \uu_n(x)h_i^{-1} \notin\uu_n(xg)B_H(e_H,C)\Big\}\Big)\\
	& = \lim_{R\to\infty}\lim_{\mathcal{U}}\PP_{\mathcal{G}_n}\Big(\Big\{x\in\mathcal{G}_n\mid \forall x' \in B_{\mathcal{G}_n}(x,R)\colon \uu_n(x)h_i^{-1}\notin B_{\mathcal{H}_n}(\uu_n(x'),C) \Big\}\Big)\\
	& \le \lim_{R\to\infty}\lim_{\mathcal{U}}\PP_{\mathcal{G}_n}\Big(\Big\{x\in\mathcal{G}_n\mid  \exists x'' \in B_{\mathcal{G}_n}(x,R)^c\colon \uu_n(x)h_i^{-1}\in B_{\mathcal{H}_n}(\uu_n(x''),C) \Big\}\Big)\\
	& \le \lim_{R\to\infty} \lim_\mathcal{U}\PP_{\mathcal{G}_n}\left(\left\{x\in \mathcal{G}_n\mid \diam_{\mathcal{G}_n}(\uu^{-1}(B_{\mathcal{H}_n}(\uu_n(x),C+|h_i|))\geq R\right\}\right)\\
	& = 0,
	\end{align*}
	where we have used for the first inequality that $B_H(e_H,C)\subset \{y_j\mid j<i\}$. For the second inequality in the fourth line, we have used that $\uu_n$ is $C$-dense. Indeed, this means that with proba tending to $1$, any $y\in \mathcal{H}_n$ lies in $B_{\mathcal{H}_n}(\uu_n(x''),C)$ for some $x''\in \mathcal{G}_n$. 
	Finally, we concluded thanks to the assumption that $\uu_n$ is $\mathcal{U}$-statistically expansive.
	\end{proof}

\section{Cocycles and their integrability}\label{sec:integrability}
In this section we calculate the length of the cocycles of the measure equivalence obtained in \cref{prop:fundDom}.
We briefly recall basic facts about our coupling from $G$ to $H$: the action of $G\times H$ on $\Omega$ is given by 
$(g,h)*\aaa=(\aaa_H(g^{-1}\cdot)h,\aaa_\N(g^{-1}\cdot))$, 
the fundamental domain for the $H$-action is given by $X_H = \{\aaa\in \Omega''\mid \aaa_H(e_G)=e_H\}$. Finally, the cocycle $\alpha\colon G\times X_H\to H$, is defined by $g*\aaa\in\alpha(g,\aaa)^{-1}*X_H$ for every $g\in G$ and almost every $\aaa\in X_H$. We deduce a very simple formula for $\alpha$:
\[\alpha(g,\aaa)=\aaa_H(g^{-1})^{-1}.\]

First, we look at the integrability of the cocycle $\alpha\colon G\times X_H\to H$. 
The integrability of this cocycle entirely depends on the convergence rate of the $\mathcal{U}$-statistically Lipschitz condition.
\begin{proposition}\label{prop:QuantSubgroup}
	Let $\varphi\colon \R^+\to \R^+$ be a non-decreasing map.
	Let $(\Omega,X_H,\mu)$ be a measure subgroup coupling from $G$ to $H$ as in \cref{thm:ME} and let $\alpha\colon G\times X_H\to H$ be associated cocycle. 
	 Then, $\alpha$ is $\varphi$-integrable if  for every $g\in S_G$ there exists a $\delta>0$ such that
	\begin{equation}\label{eq:UQuantLip}
	\sum_{r=0}^\infty \varphi\!\left(\delta r\right)\lim_{\mathcal{U}} \PP_{ \mathcal{G}_n}\left(\left\{x\in \mathcal{G}_n\mid d_{\mathcal{H}_n}(\uu_n(x),\uu_n(xg^{-1}))=r\right\}\right)<\infty.
	\end{equation}
	Similarly, $\alpha$ is strongly $\varphi$-integrable (or $\varphi^\diamond$-integrable)  if for every $\varepsilon>0$ there exists a $\delta>0$ and $C>0$ such that for every $g \in G$ we have that
	\begin{equation}\label{eq:UQuantLipStrong}
	\sum_{r=0}^\infty\varphi\!\left(\delta r\right)\lim_{\mathcal{U}} \PP_{ \mathcal{G}_n}\left(\left\{x\in \mathcal{G}_n\mid d_{\mathcal{H}_n}(\uu_n(x),\uu_n(xg^{-1}))=r\right\}\right)\le C\varphi(\varepsilon |g|).
	\end{equation}
\end{proposition}
\begin{proof}
For all $h\in H$, and $g\in G$, we have
\begin{eqnarray*}
\mu\left(\left\{\aaa\in X_H\mid \aaa_H(g)=h\right\}\right) & = & \mu\left(X_H\cap\left\{\aaa\in \Omega \mid \aaa_H(g)=h\right\}\right) \\
& = & \mu\left(\left\{\aaa\in \Omega''\mid \aaa_H(e_G)=e_H,\;  \aaa_H(g)=h\right\}\right)\\
&= &\lim_{\mathcal{U}} \PP_{\mathcal{G}_n}\left(\left\{x\in\mathcal{G}_n \mid \uu_n(xg)=\uu_n(x)h\right\}\right).
\end{eqnarray*}

Then, for every $g\in G$ we have:
	\begin{align*}
	&\int_{X_H} \varphi\!\left(\delta|\alpha(g^{-1},\aaa)^{-1}|\right) d\mu(\aaa)\\
	& = \int_{X_H} \varphi\!\left(\delta|\aaa_H(g)|\right) d\mu(\aaa)\\
	& = \sum_{h\in H} \varphi\!\left(\delta |h|\right)\mu\left(\left\{\aaa\in X_H\mid \aaa_H(g)=h\right\}\right)\\
	& =\sum_{h\in H} \varphi\!\left(\delta |h|\right)\mu\left(\left\{\aaa\in X_H\mid \aaa_H(g)=h\right\}\right)\\
	& = \sum_{h\in H}\varphi\!\left(\delta |h|\right)\lim_{\mathcal{U}}\PP_{\mathcal{G}_n}\left(\left\{x\in\mathcal{G}_n \mid \uu_n(xg)=\uu_n(x)h\right\}\right)\\
	& = \sum_{r=0}^{\infty} \sum_{|h|=r} \varphi\!\left(\delta |h|\right)\lim_{\mathcal{U}}  \PP_{\mathcal{G}_n}\left(\left\{x\in\mathcal{G}_n \mid \uu_n(xg)=\uu_n(x)h\right\}\right)\\
	& =  \sum_{r=0}^{\infty}  \varphi\!\left(\delta r\right) \lim_{\mathcal{U}}\PP_{\mathcal{G}_n}\left(\left\{x\in\mathcal{G}_n \mid d_{\mathcal{H}_n}(\uu_n(xg),\uu_n(x))=r\right\}\right)
	\end{align*}
where for the fourth equality, we have used (as in \S \ref{sec:BasicSubsets}) that \[\mu\left(\left\{\aaa\in X_H\mid \aaa_H(g)=h\right\}\right)=\overline{\mu}\left(\Big\{x\in \overline{\Omega}_0\mid T(g,x)=h\Big\}\right),\]
which equals \[\lim_{\mathcal{U}}\PP_{\mathcal{G}_n}\left(\left\{x\in\mathcal{G}_n \mid \uu_n(xg)=\uu_n(x)h\right\}\right).\]
Now both claims follow from the definitions of $\varphi$-integrable and strongly $\varphi$-integrable.
\end{proof}

Next, we look at the integrability of $\beta\colon H\times X_G\to G$. For this cocycle, its integrability depends on the convergence rate of the $\mathcal{U}$-statistically expansive condition. 
\begin{proposition}\label{prop:Cocycle}
Let $\varphi\colon \R^+\to \R^+$ be a non-decreasing map,
let $(\Omega,X_G,X_H,\mu)$ be a measure equivalence coupling from $G$ to $H$ as in \cref{thm:ME} and let $\beta\colon H\times X_G\to G$ be the measure equivalence cocycle given by $h*\aaa\in\beta(h,\aaa)^{-1}*X_G$ for every $h\in H$ and almost every $\aaa\in X_G$.
Suppose there exists a $C>0$ such that the maps $\uu_n$ are $C$-dense.
Then, $\beta$ is $\varphi$-integrable if for every $h\in B_H(e,2C+1)$ there exists a $\delta>0$ such that
	\begin{equation}\label{eq:UQuantExp}
\sum_{r=0}^\infty\varphi(\delta r) \lim_{\mathcal{U}} \PP_{\mathcal{H}_n}\left(\left\{y\in \uu_n(\mathcal{G}_n)\mid \diam_{\mathcal{G}_n}(\uu_n^{-1}(y)\cup \uu_n^{-1}(yh))\geq r\right\}\right)<\infty.
\end{equation}
Similarly, $\beta$ is strongly $\varphi$-integrable (or $\varphi^\diamond$-integrable) if  if for every $\varepsilon>0$ there exists a $\delta>0$ and $C'>0$ such that for every $h\in H$ we have that
	\begin{equation}\label{eq:UQuantExpstrong}
	\sum_{r=0}^\infty \varphi(\delta r)\lim_{\mathcal{U}} \PP_{\mathcal{H}_n}\left(\left\{y\in \uu_n(\mathcal{G}_n)\mid \diam_{\mathcal{G}_n}(\uu_n^{-1}(y)\cup \uu_n^{-1}(yh))\geq r\right\}\right)\le C'\varphi(\varepsilon |h|).
	\end{equation}
Besides, if we assume that the point-preimages of $\uu_n$ has diameter at most $C$, then these two conditions can respectively be relaxed to the following ones\footnote{where the inequality for the diameter is replaced by an equality. Note that if $\varphi$ is a quickly growing function (e.g.\ a stretched exponential), then the two conditions are equivalent, so there is no gain in assuming the bound on the diameter of preimages.}
	\begin{equation}\label{eq:UQuantExpC}
\sum_{r=0}^\infty\varphi(\delta r) \lim_{\mathcal{U}} \PP_{\mathcal{H}_n}\left(\left\{y\in \uu_n(\mathcal{G}_n)\mid \diam_{\mathcal{G}_n}(\uu_n^{-1}(y)\cup \uu_n^{-1}(yh))= r\right\}\right)<\infty.
\end{equation}
and 
	\begin{equation}\label{eq:UQuantExpCstrong}
\sum_{r=0}^\infty \varphi(\delta r)\lim_{\mathcal{U}} \PP_{\mathcal{H}_n}\left(\left\{y\in \uu_n(\mathcal{G}_n)\mid \diam_{\mathcal{G}_n}(\uu_n^{-1}(y)\cup \uu_n^{-1}(yh))= r\right\}\right)\le C'\varphi(\varepsilon |h|).
\end{equation}
\end{proposition}
\begin{proof}
First, note that by \cref{prop:fundamentalDomainBis} we have that $\mu(X_G\setminus B_H(e_H,C)*X_H)=0$, so $\mu(Y_i)=0$ for $|h_i|>C$. We let $I$ be the largest integer such that $|h_I|\le C$.

Next, note that for every $h\in H$ and $\aaa\in X_G$ there exists an $i\le I$ such that $h*\aaa\in \beta(h,\aaa)^{-1}*Y_i$. So, $\aaa(\beta(h,\aaa)^{-1})=(h^{-1}h_i,0)$. 
Besides, recall that by definition of $Y_i$, we have \[Y_i\subset h_i*Y_0=\{\aaa\in \Omega''\mid  \aaa(e_G)=(h_i,0)\}.\]
Using $X_G=\bigsqcup_{i=0}^\infty Y_i$, and arguing as in \S \ref{sec:BasicSubsets}, we have, denoting $h_{ij}=h_j^{-1}h^{-1}h_i$:
\begin{align*}
\int_{X_G}\hspace{-2pt} \varphi\!\left(\delta|\beta(h,\aaa)|\right)d\mu(\aaa) & = \sum_{g\in G} \varphi\!\left(\delta|g|\right)\mu\left(\{\aaa\in X_G\mid \beta(h,\aaa)=g\}\right)\\
& = \sum_{i,j\leq I}\sum_{g\in G} \varphi\!\left(\delta|g|\right) \mu\left(\{\aaa\in Y_j\mid \aaa(g^{-1})=(h^{-1}h_i,0)\}\right)\\
& \le  \sum_{i,j\leq I}\sum_{g\in G} \varphi\!\left(\delta|g|\right)\mu\left(\{\aaa\in \Omega''\mid \aaa(e_G)=(h_j,0),\; \aaa(g^{-1})=(h^{-1}h_i,0)\}\right)\\
& =\sum_{i,j\leq I} \sum_{g\in G} \varphi\!\left(\delta|g|\right) \lim_{\mathcal{U}}\PP_{\mathcal{G}_n}\big(\{x\in\mathcal{G}_n\mid \uu_n(xg^{-1})=\uu_n(x)h_{ij}, \\
&\AlignRight{\rho_n(x)=\rho_n(xg^{-1})=0\big\}\big)}\\
& \le \sum_{i,j\leq I}\sum_{g\in G} \varphi\!\left(\delta|g|\right) \lim_{\mathcal{U}}\PP_{\mathcal{G}_n}\big(\{x\in\mathcal{G}_n\mid \uu_n(xg^{-1})=\uu_n(x)h_{ij},
\rho_n(x)=0\big\}\big)\\
& =  \sum_{i,j\leq I} \sum_{r=0}^{\infty} \varphi\!\left(\delta r\right)\sum_{|g|=r} \lim_{\mathcal{U}}\PP_{\mathcal{G}_n}\big(\{x\in\mathcal{G}_n\mid \uu_n(xg^{-1})=\uu_n(x)h_{ij},
\rho_n(x)=0\big\}\big).
\end{align*}

For every $x$ such that $\uu_n(xg^{-1})=\uu_n(x)h_{ij}$ and 
$\rho_n(x)=0$, denote $y=\uu_n(x)$. 
Note that $x\in \uu_n^{-1}(y)$ and $xg^{-1}\in \uu_n^{-1}(yh_{ij})$. 
Restricting to $x\in{\mathcal{G}_n}^{(2r)}$ we can assume that $d(x,xg^{-1})=r$. Hence we get 
\[r\le \diam(\uu_n^{-1}(y)\cup \uu_n^{-1}(yh_{ij})).\]
 If we assume that pre-images of $\uu_n$ have diameter at lost $C$, we get by triangular inequality
\[r\le \diam(\uu_n^{-1}(y)\cup \uu_n^{-1}(yh_{ij}))\le r+2C.\]
Recall that $\uu_n$ is injective in restriction  to $\{x\in\mathcal{G}_n\mid 
\rho_n(x)=0\big\}$. 
Hence denoting $\eta= \lim_{\mathcal{U}} \frac{\#\mathcal{H}_n}{\#\mathcal{G}_n}$, we deduce that
\begin{align*}
& \sum_{|g|=r} \lim_{\mathcal{U}}\PP_{\mathcal{G}_n}\big(\{x\in\mathcal{G}_n\mid \uu_n(xg^{-1})=\uu_n(x)h_{ij}, \;
\rho_n(x)=0\big\}\big)\\
& \le
\eta \lim_{\mathcal{U}}\PP_{\mathcal{H}_n}\big(\{y\in\mathcal{H}_n \mid r\le \diam_{\mathcal{G}_n}(\uu_n^{-1}(y)\cup \uu_n^{-1}(yh_{ij}))\big\}\big),
\end{align*}
and if we assume the diameter bound on the pre-images, this becomes:
\begin{align*}
& \sum_{|g|=r} \lim_{\mathcal{U}}\PP_{\mathcal{G}_n}\big(\{x\in\mathcal{G}_n\mid \uu_n(xg^{-1})=\uu_n(x)h_{ij}, \;
\rho_n(x)=0\big\}\big)\\
& \le
\eta \lim_{\mathcal{U}}\PP_{\mathcal{H}_n}\big(\{y\in\mathcal{H}_n \mid r\le \diam_{\mathcal{G}_n}(\uu_n^{-1}(y)\cup \uu_n^{-1}(yh_{ij}))\le r+2C\big\}\big).
\end{align*}
We deduce that 
\begin{align*}
& \int_{X_G}\hspace{-2pt} \varphi\!\left(\delta|\beta(h,\aaa)|\right)d\mu(\aaa)\\ & \le \eta \sum_{i,j\leq I} \sum_{r=0}^{\infty} \varphi\!\left(\delta r \right)\lim_{\mathcal{U}}\PP_{\mathcal{H}_n}\big(\{y\in\mathcal{H}_n \mid r\le \diam_{\mathcal{G}_n}(\uu_n^{-1}(y)\cup \uu_n^{-1}(yh_{ij}))\big\}\big),
\end{align*}
which implies $\varphi$-integrability and strong $\varphi$-integrability under respectively (\ref{eq:UQuantExp}) and (\ref{eq:UQuantExpstrong}). Indeed, for integrability, we only need it for $h\in S_H$, in which case each $h_{ij}$ has word length at most $2C+1$.
Under the diameter bound we get
\begin{align*}
& \int_{X_G}\hspace{-2pt} \varphi\!\left(\delta|\beta(h,\aaa)|\right)d\mu(\aaa)\\ & \le \eta \sum_{i,j\leq I} \sum_{r=0}^{\infty} \varphi\!\left(\delta r \right)\lim_{\mathcal{U}}\PP_{\mathcal{H}_n}\big(\{y\in\mathcal{H}_n \mid r\le \diam_{\mathcal{G}_n}(\uu_n^{-1}(y)\cup \uu_n^{-1}(yh_{ij}))\le r+2C\big\}\big). 
\end{align*}
By monotonicity of $\varphi$, this implies $\varphi$-integrability under (\ref{eq:UQuantExpC}). Strong integrability follows from (\ref{eq:UQuantExpCstrong}), combined with the fact that $|h|-2C\leq |h_{ij}|\leq |h|+2C$. 
 \end{proof}

\section{Proofs of theorems from the introduction}\label{sec:Proofs}

Let us start with the qualitative statements: 
Theorem \ref{thm:Subgroup} and Theorem \ref{thm:ME} are immediate corollaries of \cref{prop:SubgroupCoupling} and \ref{prop:OECoupling} respectively.

We now turn to the proofs of Theorems \ref{thm:quantSubgroup} and \ref{thm:quantME}. 
We start showing that (\ref{eq:quantLip})  implies that $\uu_n$ is $\mathcal{U}$-statistically Lipschitz. Let us argue by contradiction. Then, by \cref{prop:StatBiLGroups} there exists an $\varepsilon>0$ such that for all $C_0$ there exists $C'>C_0$ such that for some $s\in S_G$ we have that 
\[\lim_{\mathcal U}\PP_{\mathcal{G}_n}\left(\left\{x\in \mathcal{G}_n\mid d_{\mathcal{H}_n}(\uu_n(x),\uu_n(x.s))\ge C'\right\}\right)\ge \varepsilon.\]
Note that the events $\left[\{x\in \mathcal{G}_n\mid d_{\mathcal{H}_n}(\uu_n(x),\uu_n(x.s))\ge C'\}\right]_{\mathcal U}$ are disjoint, hence by $\sigma$-additivity of $\nu$, we have
\[\lim_{\mathcal U}\PP_{\mathcal{G}_n}\left(\left\{x\in \mathcal{G}_n\mid d_{\mathcal{H}_n}(\uu_n(x),\uu_n(x.s))\ge C'\right\}\right)=\sum_{r\geq C'}\lim_{\mathcal U}\PP_{\mathcal{G}_n}\left(\{x\in \mathcal{G}_n\mid d_{\mathcal{H}_n}(\uu_n(x),\uu_n(x.s))=r\}\right)\]

By (\ref{eq:quantLip}), and using that $\varphi$ is unbounded and nondecreasing,  for every $\delta>0$ we can take $C_0$ large enough such that
\begin{align*}
	\varphi\left(\delta C_0\right) \varepsilon & >  \sum_{r=0}^R \varphi\!\left(\delta r\right)\lim_{\mathcal U}\PP_{\mathcal{G}_n}\left(\{x\in \mathcal{G}_n\mid d_{\mathcal{H}_n}(\uu_n(x),\uu_n(x.s))=r\}\right)\\
	& \ge \varphi\left(\delta C'\right) \lim_{\mathcal U} \PP_{\mathcal{G}_n}\left(\{x\in \mathcal{G}_n\mid d_{\mathcal{H}_n}(\uu_n(x),\uu_n(x.s))\ge C'\}\right)\\
	& \ge \varphi\left(\delta C'\right) \varepsilon,
\end{align*}
leading to the desired contradiction. 
 This is enough to end the proof of Theorem \ref{thm:quantSubgroup}. Indeed, we have seen that it is $\mathcal U$-statistically Lipschitz, and since it has pre-images of diameter  at most $C$, it satisfies the assumptions of \cref{prop:SubgroupCoupling}, which provides a measure subgroup subgroup coupling from $G$ to $H$. We then conclude thanks to \cref{prop:QuantSubgroup}

To prove  Theorem \ref{thm:ME}, we need to show that the sequence $(\uu_n)_n$ is $\mathcal U$- statistically expansive. The proof is similar: suppose it is not, then, by \cref{prop:expsofic}, there exists an $\varepsilon>0$ and an $h\in B_H(e_H,2C+1)$ such that for all $C_0$ there exists $C'>C_0$ such that
\[\lim_{\mathcal U}\PP_{\mathcal{H}_n}\left(\left\{z\in \uu_n(\mathcal{G}_n)\mid zh\in \uu_n(\mathcal{G}_n), \diam_{\mathcal{G}_n}(\uu_n^{-1}(z)\cup \uu_n^{-1}(zh))\ge C'\right\}\right)\ge \varepsilon.\]
Since $\varphi$ is unbounded, for any $\delta>0$,  we can take $C_0$ such that
\begin{align*}
\varphi\left(\delta C_0\right) \varepsilon & >\sum_{r=0}^\infty \varphi\!\left(\delta r\right) \lim_{\mathcal U} \PP_{\mathcal{H}_n}\left(\{z\in \uu_n(\mathcal{G}_n)\mid zh \in \uu_n(\mathcal{G}_n),  \diam_{\mathcal{G}_n}(\uu_n^{-1}(z)\cup \uu_n^{-1}(zh))=r\}\right)\\
& \ge \varphi\left(\delta C'\right)\lim_{\mathcal U}\PP_{\mathcal{H}_n}\left(\{z\in \uu_n(\mathcal{G}_n)\mid zh\in \uu_n(\mathcal{G}_n),  \diam_{\mathcal{G}_n}(\uu_n^{-1}(z)\cup \uu_n^{-1}(zh))\ge C\}\right)\\
& \ge \varphi\left(\delta C'\right) \varepsilon.
\end{align*}
We deduce that under the assumptions of \cref{thm:quantME}, the sequence $(\uu_n)$ is $\mathcal U$-statistically a coarse equivalence. Hence \cref{prop:OECoupling} provides us with an ME coupling (OE if we assume in addition that $\uu_n$ is bijective). The conclusion now follows from \cref{prop:Cocycle}.

\begin{remark}\label{rem:Strong}
	In \cref{thm:quantSubgroup} (and \cref{thm:quantME}), we obtain {\bf strong} integrability if we assume the stronger version of (\ref{eq:quantLip}): that for every $\varepsilon>0$ there exists a $\delta>0$ and $C'>0$ such that
	\[\sum_{r=0}^\infty \varphi\!\left(\delta r\right)\lim_{\mathcal U}\PP_{\mathcal{G}_n}\left(\{x\in \mathcal{G}_n\mid d_{\mathcal{H}_n}(\uu_n(x),\uu_n(xg))=r\}\right)\le C'\varphi(\varepsilon|g|)\]
	for every $g\in G$. Similarly the coupling is $(\varphi,\psi^\diamond)$-integrable if for every $\varepsilon>0$ there exists a $\delta>0$ and $C'>0$ such that
	\begin{equation}\label{eq:quantExpStrong}
	\sum_{r=0}^\infty \psi\!\left(\delta r\right)\lim_{\mathcal U} \PP_{\mathcal{H}_n}\left(\{y \in \uu_n(\mathcal{G}_n)\mid \diam_{\mathcal{G}_n}(\uu_n^{-1}(y)\cup\uu_n^{-1}(yh)) = r\}\right)\le C'\psi(\varepsilon |h|)\end{equation}
	for every $h\in H$ and it is $(\varphi^\diamond,\psi^\diamond)$-integrable if both conditions are satisfied.
	\end{remark}

\section{Application to the class $\mathcal M$}\label{sec:M}
This section is dedicated to the proof of \cref{thm:SolvableExpME}. Our main contribution is the following more specific theorem.

\begin{theorem}\label{thm:example}
	Let $k\ge 2$ and  $A\in \GL_2(\Z)$ has a real eigenvalue $\lambda>1$. 
 Then there exists a mutually cobounded $(\LL^\infty,\exp^\diamond)$-measure equivalence coupling from $\Lamp_k$ to $\SOL_A=\Z^2\rtimes_A\Z$.
\end{theorem}
\begin{proof}
Recall that according to our notation, $\Lamp_k=\Z/k\Z\wr \Z$. Let $\pi:\Lamp_k\to \Z$ be the projection to $\Z$. Given $q\geq 1$, the pre-image of $q\Z$ by $\pi$ is isomorphic to $(\Z/k\Z)^q\wr \Z$, which itself is bi-Lipschitz equivalent to $\Lamp_{k^q}$. In particular, $\Lamp_{k^q}$ and $\Lamp_k$ are quasi-isometric. Hence in order to prove the theorem, it is harmless to assume that $k$ is arbitrarily large. For reasons that will appear later (which are not essential), we will assume from now on that $k\geq \lambda$.

We will use the fact that $\SOL_A$ is a uniform lattice in the group $\SOL_\R=\R^2\rtimes \R$ where $\R$ acts by $t\cdot (x,y)=(k^tx,k^{-t}y)$. 

The proof of the theorem goes in three steps: first we construct a sequence of maps between F\o lner sequences of $\Lamp_k$ and  $\SOL_\R$ (by far the main step of the proof). Actually we will construct a unique map between the groups and consider its restriction to a F\o lner sequence. Second,  we will perturb this map slightly (at bounded distance) to have it land in a copy of $\SOL_A$ embedded as a lattice in $\SOL_\R$. 
Along the way all the metric estimates will be in terms of the word metrics of the groups. Hence a last technical step will be required to translate them in terms of the intrinsic graph distances of the F\o lner sets.

\

\noindent{\bf First step: constructing a map from $\Lamp_k$ to $\SOL_\R$.}

We start defining generating sets for the groups $\Lamp_k$ and $\SOL_\R$. 
We let $S^\Lamp=S_1^\Lamp\cup S_2^\Lamp$, where $S_1^\Lamp$ consists of elements $(x,0)$, where $x$ is supported at $0$, and $S_2^\Lamp=\{(0,\pm 1\}$.
Similarly we let $S^{\SOL}=S_1^{\SOL}\cup S_2^{\SOL}$, where $S_1^{\SOL}$ consists of elements $(x,0)$, where $\|x\|_\infty\leq 1$, and $S_2^{\SOL}=\{(0,t\}$, with $|t|\leq 1$. We shall denote by $|g|$ the word length of an element of either $\Lamp_k$ or $\SOL_\R$.
We shall use repeatedly the following easy estimates, whose verifications are left to the reader.
 For $g=(x,j)\in \Lamp_k$, 
\[\max \left\{\diam(\supp x\cup \{0\}),|j|\right\} \leq |g|\leq 2\diam(\supp x\cup \{0\})+|j|.\]
and for $g=(x,j)\in \SOL_\R,$
 \[\max \left\{\log_k\|x\|_{\infty},|j|\right\} \leq |g|\leq 2\log_k(1+ \|x\|_{\infty})+2|j|.\]

We consider the following F\o lner sequences of $\Lamp_k$ and $\SOL_\R$: 
\[F_n^\Lamp=\left\{(x,m)\mid \supp x\subset [-n,n], \; 0\leq m\leq n\right\},\] 
and 
\[F_n^{\SOL}=\left\{(x,m)\in \R_+^2\times \R\mid \|x\|_\infty \leq k^{n+1}, \; 0\leq m\leq n\right\}.\] 
Let us check that these are indeed right F\o lner sequences. For both, almost invariance under right multiplication by $S_2^\Lamp$ (resp.\ $S_2^{\SOL}$) is true because it only changes the second coefficient. For $F_n^\Lamp$, we observe that it is invariant under right multiplication by $S_1^\Lamp$, so it is duly a F\o lner sequence. Note that right multiplication by $S_1^{\SOL}$ preserves each coset of $\R^2$. Besides, for every $j\in \R$, we have 
\[([0,k^{n+1}]^2\times\{j\}) S_1^{\SOL}\subset [-k^j,k^{n+1}+k^j]^2\times\{j\}. \]
The Haar measure $\lambda$ on  $\SOL_\R$ is simply the product of the Lebesgue measure on $\R^3$.
Therefore, we have 
\[\lambda(([0,k^{n+1}]^2\times\{j\}) S_1^{\SOL})\leq (k^{n+1}+2k^j)^2,\]
so
\[\lambda(F_n^{\SOL}S_1^{\SOL})\leq \int_{0}^n (k^{n+1}+2k^j)^2dj=nk^{2n+2}+4k^{n+1}\int_{0}^nk^jdj+2\int_{0}^nk^{2j}dj\leq (n+6)k^{2n+2},\]
And since $\lambda(F_n^{\SOL})=nk^{2n+2}$, we deduce that 
\[\frac{\lambda(F_n^{\SOL}S_1^{\SOL})}{\lambda(F_n^{\SOL})}\leq \frac{n+6}{n},\]
which tends to $1$ as $n\to \infty$. Hence $F_n^{\SOL}$ is a F\o lner sequence.

\

We now define a map $\vv:\Lamp_k\to \SOL_\R$  as follows. For $x=(\eps_i)\in \oplus_\Z \Z/k\Z$, we set
\[\vv(x,m)=\left(\sum \eps_i k^i, \sum \eps_i k^{-i},m\right).\]

Before analyzing the properties of this map, let us pause, and work out  the following baby case as a warmup. Consider the function $f:\oplus_n \Z/k\Z\to \R_+$ defined  for each $x=(\eps_i)$ by $f(x)=\sum \eps_i k^i$. Suppose we equip the group  $\oplus_n \Z/k\Z$ with the length $|x|=\max (\supp x)$, and  $\R_+$ is equipped with the distance $d(x,y)=\log (1+|x-y|)$. Note that this is similar to the restriction of $\vv$ to $\bigoplus_\Z \Z/k\Z$.  It is easy to see that the map $f$ is injective, Lipschitz, and takes values in $\N$. Besides, for all $n\in \N$, $f$ maps bijectively $K_n=\{x\mid \supp x\subset [0,n]\}$ to $I_n=[0, k^{n+1}-1]\cap \N$. By contrast, this map is not coarsely expansive. This results from the fact that $\R$ is Archimedean, which manifests itself through carries: assume we add $1$ to $x=\sum \eps_i k^i$, and then write the $k$-adic decomposition of the resulting number, we get $x+1=\sum_i \eps_i'$, where $\eps_i'\neq \eps_i$ for all $i\leq m$, where $m$ is the smallest integer such that $\eps_i\neq k-1$. Hence the pre-image of $x+1$ lies at distance $m$ from $x$, and $m$ can be arbitrarily large. However, large values of $m$ happen ``exponentially rarely''. Let us indeed evaluate the proportion of elements of $x\in K_n$ such that $d(x,f^{-1}(f(x)+1))\geq m$: since it imposes that the first $m-1$ coordinates of $x$ are equal to $k-1$, this proportion is at most $k^{-m+1}$. 

We are now ready to study our map $\vv$, for which a similar phenomenon occurs, although in a more complicated way, as shown by the following lemma.

\begin{lemma}\label{lem:psi}
The map $\vv$ is $2$-Lipschitz, and maps $F_n^\Lamp$ to a $(1+\log_k (1+2k))$-coarsely dense subset of $F_n^{\SOL}$. 
It satisfies the following (statistic expansivity) property: for all integers $q,m\geq 1$,
the proportion of points $g\in F_n^\Lamp$ such that $\diam\vv^{-1}(B_{\SOL}(\vv_n(g),q))\geq 2m+3q$ is at most $4k^{-m+1}$. Finally, For every $r$, the preimage of balls of radius $q$ has cardinality at most $(2q+1)k^{2q+1}$.
\end{lemma}

\begin{proof}
Let us prove Lipschitz: for this it is enough to consider neighboring vertices: for those who differ by their $\Z$-coordinate, the statement is obvious, so consider the case of two elements $g,g'\in \Lamp_k$ such that $g'=gs$ for some $s\in S_1^{\Lamp}$. We then have $g=(x,j)$ and $gs=(x',j)$ where $x=(\eps_i)$ and $x'=(\eps'_i)$ only differ in their $j$-coordinate. We have $\vv(g)=(y,j)$, where $y=(\sum_i \eps_i k^i, \sum_i \eps_i k^{-i})$ and $\vv(g')=(y',j)$ with $y'=(\sum_i \eps'_i k^i, \sum_i \eps'_i k^{-i})$. So
\[\vv(g)^{-1}\vv(g')=\left(\sum_i(\eps'_i-\eps_i) k^{i-j},\sum_i(\eps'_i-\eps_i) k^{j-i},0\right)=(\eps'_j-\eps_j,\eps_j-\eps'_j,0),\] 
which has length at most $2$. Hence the map is $2$-Lipschitz.

Let us prove that $\vv$ maps $F_n^\Lamp$ to a coarsely dense subset of $F_n^{\SOL}$. The fact that $\vv(F_n^\Lamp)\subset F_n^{\SOL}$ is obvious, so we focus on the coarse density. We shall prove that every element $h=(a,b,j)\in F_n^{\SOL}$ lies at bounded distance from $\vv(F_n^\Lamp)$. Perturbating $h$ at distance at most one, we can assume that $j\in \Z$. Let $g=(x,j)\in \Lamp_k$, with $x=(\eps_i)$.
We have $\vv(g)=\left(\sum_i\eps_ik^i,\sum_i \eps_ik^{-i}, j\right)$. Hence,
\[\vv(g)^{-1}h=\left(\left(k^{-j}a-\sum_i \eps_ik^{i-j}\right), \left(k^jb-\sum_i \eps_ik^{-i+j}\right), 0\right).\] 
Since $h\in  F_n^{\SOL}$, we have $\lfloor k^{-j}a\rfloor\leq k^{n-j+1}$ and $\lfloor k^{j-1}b\rfloor\leq k^{n+j}$, so their  $k$-adic decompositions are respectively supported in $[0,n-j]$ and $[0,n+j-1]$. Now, choose $\eps_i$ be such that $\sum_{i\geq j} \eps_ik^{i-j}$ is the $k$-adic decomposition of $\lfloor k^{-j}a\rfloor$, and  $\sum_{i< j} \eps_ik^{-i+j-1}$ is the $k$-adic decomposition of $\lfloor k^{j-1}b\rfloor$. It follows that $\eps_i$ is supported on $[-n,n]$, hence that $g\in F_n^\Lamp$. 
Besides, \[\left|k^{-j}a-\sum_{i\geq j} \eps_ik^{i-j}\right|=k^{-j}a-\lfloor k^{-j}a\rfloor
\leq 1\] and  \[\left|k^jb-\sum_{i<j} \eps_ik^{-i+j}\right|=k\left|k^{j-1}b-\sum_{i<j} \eps_ik^{-i+j-1}\right|=k(k^{j-1}b-\lfloor k^{j-1}b\rfloor)\leq k.\]
We therefore have $\left|k^{-j}a-\sum_{i} \eps_ik^{i-j}\right|\leq 2$ and $\left|k^jb-\sum_{i} \eps_ik^{-i+j}\right|\leq 2k$, and so we have duly constructed an element $g\in F_n^\Lamp$ such that $d(\vv(g),h)\leq \log_k (1+2k)$.

We now turn our attention to the last two properties, which will be proved all at once. 
For this we fix some $g\in F_n^\Lamp$.
Let $g'\in \Lamp_k$ such that $d(\vv(g'), \vv(g))\leq q.$ Note that if $g=(x,j)$ and $g'=(x',j')$, we have \[g^{-1}g' = ((-j)\cdot (x'-x), j'-j).\]
Write $x=(\eps_i)$ and $x'=(\eps'_i)$. We have
\[\vv(g)^{-1}\vv(g')=\left(\sum_i(\eps'_i-\eps_i) k^{i-j},\sum_i(\eps'_i-\eps_i) k^{j-i},j'-j\right).\]
The condition that $d(\vv(g'), \vv(g))\leq q$ implies that $|j-j'|\leq q$. So $d(g,g')\geq 2m+3q$ implies that $2\diam((\supp(x'-x)-j)\cup \{0\})+q\geq 2m+3q$, hence  $\diam((\supp(x'-x)-j)\cup \{0\})\geq m+q$.
We let $\ell_+,\ell_-$ be respectively the maximal, minimal integer such that $\eps_{\ell_+}\neq \eps'_{\ell_+}$,  $\eps_{\ell_-}\neq \eps'_{\ell_-}$.
Note that either $\ell_+\geq  j+m+q$ or $\ell_-\leq j-m-q$.
Let $J_-=[\ell_-+1,j-q-1]$ and $J_+=[j+q+1,\ell_+-1]$. Note that one of these intervals can possibly be empty, but one of them has size at least $m-1$.
\begin{clai}\label{clai:J}
In each interval $J_-$ and $J_+$, either $\eps_i$ is identically $0$, and $\eps'_i$ is identically $k-1$, or vice versa.
\end{clai}
\begin{proof}
Both cases are treated similarly, so we only consider $J_+$, which we can assume to be non empty.  
We have
\[\sum_i(\eps'_i-\eps_i) k^{i-j}=(\eps'_{\ell_+}-\eps_{\ell_+})k^{{\ell_+}-j}+\sum_{i<\ell_+}(\eps'_i-\eps_i)k^{i-j}.\]
By symmetry, we can assume without loss of generality that $\eps'_{\ell_+}>\eps_{\ell_+}$.
Therefore,
\[\left|\sum_i(\eps'_i-\eps_i) k^{i-j}\right|=\sum_i(\eps'_i-\eps_i) k^{i-j}\geq k^{\ell_+-j}+\sum_{i<\ell_+}(\eps'_i-\eps_i)k^{i-j}.\]
We let $t$ be the maximal integer  $<\ell_+$ such that $\eps'_{t}-\eps_{t}\neq -k+1$. We have
\[\left|\sum_i(\eps'_i-\eps_i) k^{i-j}\right|\geq k^{\ell_+-j}-(k-1)\sum_{i=t}^{\ell_+-1}k^{i-j}=k^{t-j}.\]
On the other hand, since $d(\vv(g'), \vv(g))\leq q$, we have $t-j\leq q$. In other words, $\eps_i'-\eps_i=-k+1$ for all $r+j<i< \ell_+$, hence for all $r+j<i<m+q+j$. This imposes for such values of $i$ that $\eps_i=k-1$ and $\eps'_i=0$, and the claim is proven.
\end{proof}
It is now straightforward to deduce our two properties from the claim. First, since the coordinates of $x$ on $J_+$ and $J_-$ are prescribed, with two possible values, the proportion of such $g\in F_n^\Lamp$ is at most $4k^{-m+1}$. 
Second, if $x$ is given, we see that the coordinates of $x'$ outside the interval $[j-q,j+q]$ are determined by those of $x$. This leaves $k^{2q+1}$ possibilities for choosing $x'$, and since $|j'-j|\leq q$, we have $(2q+1)k^{2q+1}$ for the choice of $g$. This ends the proof of the lemma.
\end{proof}

\noindent{\bf Second step: constructing a map from $\Lamp_k$ and  $\SOL_A$.} 

Let us describe an embedding of  $\SOL_A$ in  $\SOL_\R$. We start with the obvious embedding of $\SOL_A$ in the semi-direct product $\R^2\rtimes_A \R$, where the notation means that $t\in \R$ acts by multiplication by the matrix $A^t$. Denote by $\Gamma_0'$ the image of $\SOL_A$. A relatively compact fundamental domain $D_0'$ for the action by left-translations of $\Gamma_0$ on $\R^2\rtimes_A \R$ is simply given by $[0,1)^3$.
Choose a basis $(v_+,v_-)$ of eigenvectors of $A$ associated to the eigenvalues $\lambda$ and $\lambda^{-1}$. We have an isomorphism $ \SOL_\R\to \R^2\rtimes_A \R$ given by \[(a,b,t)\mapsto (av_++bv_-, \frac{\log \lambda}{\log k}t),\]
whose inverse will be denoted by $\mathfrak{p}$. Denote $\Gamma_0=\mathfrak{p}(\Gamma_0')$ and  $D_0=\mathfrak{p}(D_0')$.  By construction, $\Gamma_0$ is a copy of $\SOL_A$ embedded as a lattice in $\SOL_\R$, and $D_0$ is a fundamental domain for its action by left-translations. Note that $D_0$
 is relatively compact, and with no-empty interior and contains the neutral element.

Although $\vv$ is obviously injective, its image is a priori not uniformly discrete\footnote{It is indeed  not the case as the elements  $(x,0)$ and $(x',0)$, where $x=(k-1)1_{[-m,m]}$ and $x'=1_{\{-m-1,m+1\}}$ are at distance  $O(q^{-m})$ from each other.}.
We start modifying $\vv$ so that while remaining injective, its image is a uniformly discrete subset of $\SOL_\R$. By Lemma \ref{lem:psi},  the preimage of any (left) $\Gamma_0$-translate of $D_0$ has cardinality bounded above by some constant $N$. In other words, every such translate of $D_0$ contains at most $N$ elements of $\Im\vv$. Since $D_0$ has non-empty interior, we can therefore modify $\vv$ to a map $\vv_0$ such that for all $g$, $\vv_0(g)$ belongs to the same $\Gamma_0$-translate of $D_0$ as $\vv(g)$, and such that $\vv_0$ has uniformly discrete image.

Our next step consists in modifying $\vv'$ further to have it land in some (different) copy of $\SOL_A$ in $\SOL_\R$.  Our goal is to do it in order to get an injective map from $\Lamp_k$ to $\SOL_A$. 
Since we have assumed (from the beginning) that $k\geq \lambda$, we already have that two elements of $D_0\cap \Im\vv$ have same $\R$-coordinate, which ensures that $\vv_0$ is already injective in restriction to the $\Z$ coordinates.  In order to obtain an injective map, we will replace $\Gamma_0$ by a copy which is ``denser' in the $\R^2$-coordinate.

For all $t>0$, consider the automorphism $\theta_t$ which multiplies the $\R^2$-coordinate by $t$. Let $\tau_t=\theta_t\circ \tau_0$, $\Gamma_t=\Im(\tau_t)$ and let $D_t=\theta_t(D_0)$. Observe that $D_t$ is a fundamental domain for the left $\Gamma_t$ action on $\SOL_\R$. Hence by choosing $t$ small enough, we can assume that the euclidean diameter of $D_t\cap \R^2$ is arbitrarily small. Combined with the fact that $\Im\vv_0$ is uniformly discrete, and that  $\theta_t$ does not affect the $\R$ coordinate, we can choose $t\leq 1$ small enough so  that  every $\Gamma_t$-translate of $D_t$  contains at most one element of $\Im\vv_0$. From now on, we fix such $t$ and omit the subscript $t$.

We define a map $\uu:\Lamp_k\to \Gamma$ as follows: $\uu(g)=\gamma\in \Gamma$ where $\gamma$ is the unique element such that $\vv_0(g)\in \gamma D$. Since $\gamma D$ contains at most one element of $\Im\vv_0$, the map $\uu$ is injective.

Now, let $C$ be such that $D_0\subset B_{\SOL_\R}(e,C/2)$. Note that we also have  $D\subset B_{\SOL_\R}(e,C/2)$. 
By construction,  $d(\uu(g),\vv(g))\leq C$ for all $g\in \Lamp_k$ (where the distance is measured with respect to the word metric of $\SOL_\R$). 

We let $\mathcal G_n=F_n^\Lamp$ and $\mathcal H_n=F_n^{\SOL}B_{\SOL_\R}(e, C)\cap \Gamma$. Since $\vv$ maps $F_n^\Lamp$ to a coarsely dense subset of $F_n^{\SOL}$, we have that $\uu$ maps it to a coarsely dense subset of $\mathcal H_n$. Restricting $\uu$ to $\mathcal G_n$, we therefore get a sequence of maps $\uu_n: \mathcal G_n\to \mathcal H_n$.

Let us check that $\mathcal H_n$ is a F\o lner sequence of $\Gamma$. Let $R$ be such that $T=B_{\SOL_\R}(e,R)\cap \Gamma$ defines a generating set of $\Gamma$. Then it is enough to prove that $\lim_{n\to \infty}\frac{\#\mathcal{H}_nT}{\#\mathcal{H}_n}=1$.  Note that
\[\frac{\#\mathcal{H}_nT}{\#\mathcal{H}_n}=\frac{\lambda(\mathcal{H}_nTD)}{\lambda(\mathcal{H}_nD)}.\]
On the other hand, suppose $g\in \SOL_\R$, then it belongs to a unique $\gamma D$, where $\gamma\in \Gamma$. In particular we have $d(g,\gamma)\leq C.$ So if $g\in F_n^{\SOL}$, then $\gamma\in \mathcal{H}_n$.
Hence,  $F_n^{\SOL}\subset \mathcal{H}_nD$, which implies that $\lambda(\mathcal{H}_nD)\geq \lambda(F_n^{\SOL})$.
Now, we have that $TD\subset B(e,R+C)$, hence $\mathcal{H}_nTD\subset F_n^{\SOL}B(e,R+2C)$. We deduce that 
\[\frac{\#\mathcal{H}_nT}{\#\mathcal{H}_n}\leq \frac{\lambda(F_n^{\SOL}B(e,R+2C))}{\lambda(F_n^{\SOL})},\]
which converges to $1$ since $F_n^{\SOL}$ is a F\o lner sequence.

We fix an arbitrary finite generating set for $\Gamma$.
Observe that since $\Gamma$ is a uniform lattice in $\SOL_\R$, its word metric is bi-Lischitz equivalent to the restriction the word metric of $\SOL_\R$.
We deduce from Lemma \ref{lem:psi} and the previous discussion that $\uu$ satisfies the following properties.

\begin{lemma}\label{lem:phi}
There exists a constant $L\geq 1$ such that $\uu$ satisfies the following properties:
\begin{itemize}
\item[(i)] it is $L$-Lipschitz;
\item[(ii)] it is injective;
\item[(iii)] it has $L$-dense image;
\item[(iv)] for all integers $q,m\geq 1$,
the proportion of points $g\in \mathcal G_n$ such that $\diam\uu^{-1}(B_{\Gamma}(\uu(g),q/L-L))\geq 2m+3q$ is at most $4k^{-m+1}$.
\end{itemize}
\end{lemma}
Our next step is the following consequence of (ii) and (iv) from the previous lemma.
\begin{lemma}\label{lem:Last}
There exists a constant $C'$ such that for all $\delta>0$, there exists $C'''$ such that for all $h\in \Gamma$,
\[\sum_{r\geq 0}e^{\delta r}\PP_{ \mathcal{H}_n}\left(\left\{y\in \uu( \mathcal{G}_n) \mid  \diam(\uu^{-1}(y)\cup \uu^{-1}(yh))\geq r \right\}\right)\leq C'''e^{C'\delta |h|}.\]
\end{lemma}
\begin{proof}
Reformulating (iv), we have 
\[ \PP_{ \mathcal{G}_n}\left(\left\{g \mid \diam\uu^{-1}(B_{\Gamma}(\uu(g),q/L-L))\geq 2m+3q\right\}\right)\leq 4 k^{-m+1}.  \]
By (ii), this implies that 
\[\PP_{ \mathcal{H}_n}\left(\left\{y\in \uu( \mathcal{G}_n)\mid \diam\uu^{-1}(B_{\Gamma}(y,q/L-L))\geq 2m+3q\right\}\right)\leq 4 k^{-m+1}.\]
We set $p=q/L-L$. Then the previous inequality implies the following one:
\[\PP_{ \mathcal{H}_n}\left(\left\{y\in \uu( \mathcal{G}_n)\mid \diam\uu^{-1}(B_{\Gamma}(y,p))\geq 2m+3pL+L^2p\right\}\right)\leq 4 k^{-m+1}.\]
Letting $r=2m+3pL+L^2p$, and $C=(3L+L^2)/2$ and we deduce 
\[\PP_{ \mathcal{H}_n}\left(\left\{y\in \uu( \mathcal{G}_n)\mid \diam\uu^{-1}(B_{\Gamma}(y,p))\geq r \right\}\right)\leq 4 k^{-r/2+Cp+1}.\]
Hence we deduce that for all $h\in \Gamma$, 
\[\PP_{ \mathcal{H}_n}\left(\left\{y\in \uu( \mathcal{G}_n)\mid  \diam(\uu^{-1}(y)\cup \uu^{-1}(yh))\geq r \right\}\right)\leq 4 k^{-r/2+C|h|+1}.\]
Let $r_0=4C|h|+4$.
Then for $r\geq r_0$, we have
\[\PP_{ \mathcal{H}_n}\left(\left\{y\in \uu( \mathcal{G}_n) \mid  \diam(\uu^{-1}(y)\cup \uu^{-1}(yh))\geq r \right\}\right)\leq 4 k^{-r/4}.\]
Besides for all $\delta>0$, $|h|\geq 1$
\[\sum_{r=0}^{r_0-1}e^{\delta r}\leq \frac{e^{\delta(4C|h|+4)}}{e^{\delta}-1}\leq C''e^{C'\delta |h|},\] where $C'=4C+4$, and $C''=1/(e^\delta-1)$.
We can now conclude that for $\delta\leq (\log k)/8$.
\begin{align*}
\sum_{r\geq 0}e^{\delta r}[\PP_{ \mathcal{H}_n}\left(\left\{y\in\uu( \mathcal{G}_n)\mid  \diam(\uu^{-1}(y)\cup \uu^{-1}(yh))\geq r \right\}\right)
& \leq C''e^{C'\delta |h|} +4\sum_{r=r_0}^{\infty}k^{-r/4} e^{\delta r}\\
& \leq C''e^{C'\delta |h|} +4 e^{\delta r_0} \sum_{r=0}^{\infty} k^{-r/4} e^{\delta r}\\
& \leq C''e^{C'\delta |h|} +4 e^{C'\delta |h|} \sum_{r=0}^{\infty} k^{-r/8}\\
& \leq C'''e^{C'\delta |h|},
\end{align*}
where $C'''=C''+4 \sum_{r=0}^{\infty} k^{-r/8}$, and the lemma is proven.
\end{proof}

\noindent{\bf Last step: from word distance to intrinsic distance.}

Note that in Lemmas \ref{lem:phi} and \ref{lem:Last}, we use the distance in the groups and not the intrinsic graph distance in the F\o lner sets. Obviously the last one is larger or equal than the first one, and it coincides for elements $x,y$ at distance at most $r$ in $\mathcal{H}_n$, such that $x\in \mathcal{H}_n^{(r)}$. Actually here,  $\mathcal{H}_n^{(r)}$ contains all vertices of $\mathcal{H}_n$ which are at $\Gamma$-distance at least $r+1$ from the complement of $\mathcal{H}_n$. 

Passing to the intrinsic distance, we might loose $L$-density of the map $\uu: \mathcal{G}_n\to \mathcal{H}_n$. But it can be restored by modifying slightly $\mathcal H_n$ as follows: replace it by $\mathcal H_n'=\mathcal H_nB_\Gamma(e_\Gamma,L)$ (which is still a F\o lner seqence). Now $\mathcal H_n$ lies in $\mathcal{H}_n'^{(L)}$ and by  the previous remark,
$\uu(\mathcal{G}_n)$ is $L$-dense in $\mathcal H_n$ for the {\it intrinsic distance} in $\mathcal H_n'$. But by construction $\mathcal H_n$ is $L$-dense in $\mathcal H'_n$ for the intrinsic distance: indeed the latter equals $\bigcup_{y\in \mathcal H_n}B_\Gamma(y,L)$.
Hence we deduce that $\uu(\mathcal{G}_n)$ is $2L$-dense in $\mathcal H_n'$ for the intrinsic distance of $\mathcal H_n'$. 
A similar argument shows that $\uu: \mathcal{G}_n\to \mathcal{H}'_n$ is $L$-Lipschitz.

To get a $\mathcal U$-statistical version of \cref{lem:Last} in intrinsic distance, we use the following lemma.
\begin{lemma}\label{lem:intrinsic}
We have
\[\lim_{\mathcal U}\PP_{ \mathcal{H}'_n}\left(\left\{y\in\uu( \mathcal{G}_n)\mid  \diam(\uu^{-1}(y)\cup \uu^{-1}(yh))= \diam_{\mathcal H_n'}(\uu^{-1}(y)\cup \uu^{-1}(yh))\right\}\right)=1. \]
\end{lemma}
\begin{proof}
Denote $r_y= \diam(\uu^{-1}(y)\cup \uu^{-1}(yh))$. 
By the previous discussion it suffices to prove that 
\begin{equation}\label{eq:r_y1}
\lim_{\mathcal U}\PP_{ \mathcal{H}'_n}\left(\left\{y\in\uu( \mathcal{G}_n)\mid \uu^{-1}(y)\in \mathcal{G}_n^{r_y}\right\}\right)=1.
\end{equation}
By \cref{lem:Last}, we know that 
\begin{equation}\label{eq:r_y2}
\lim_{r\to \infty}\lim_{\mathcal U}\PP_{ \mathcal{H}'_n}\left(\left\{y\in\uu( \mathcal{G}_n)\mid  r_y\leq r\right\}\right)=1.
\end{equation}
Denoting $\eta=\lim_{\mathcal U}|\mathcal G_n|/|\mathcal H'_n|$, we have, for each $r\geq 0$
\[\lim_{\mathcal U}\PP_{ \mathcal{H}'_n}\left(\left\{y\in\uu( \mathcal{G}_n)\mid \uu^{-1}(y)\notin \mathcal{G}_n^{(r)}\right\}\right)=\eta\PP_{ \mathcal{G}_n}\left(\left\{x\in \mathcal{G}_n\mid x\notin \mathcal{G}_n^{(r)}\right\}\right)=0.\]
Therefore, 
\[\lim_{\mathcal U}\PP_{ \mathcal{H}'_n}\left(\left\{y\in\uu( \mathcal{G}_n)\mid \uu^{-1}(y)\in \mathcal{G}_n^{(r)}\right\}\right)=1,\]
which combined with (\ref{eq:r_y2}) implies (\ref{eq:r_y1}). This proves the lemma. 
\end{proof}

Let us consider the sequence of maps $\uu_n: \mathcal{G}_n\to \mathcal{H}'_n$, induced by $\uu$. By the previous discussion, $\uu_n$ is injective, $L$-Lipschitz, has $2L$-dense image. 
Besides, combining \cref{lem:Last} and \cref{lem:intrinsic}, we deduce that $\uu_n$ satisfies 
\[\sum_{r\geq 0}e^{\delta r}\lim_{\mathcal U}\PP_{ \mathcal{H}'_n}\left(\left\{y\in \uu( \mathcal{G}_n) \mid  \diam_{\mathcal H'_n}(\uu^{-1}(y)\cup \uu^{-1}(yh))\geq r \right\}\right)\leq C'''e^{C'\delta |h|}.\]

\noindent{\bf Concluding step.}

 We are now ready to end the proof of Theorem \ref{thm:example}. 
The conditions of \cref{thm:quantME}  and Remark \ref{rem:Strong} being satisfied\footnote{Recall that the $L^{\infty}$ condition is formally equivalent to $\varphi$-integrability for a function $\varphi:\R_+\to \R_+\cup \{+\infty\}$, where $\varphi(t)=\infty$ for $t\geq 1$.}, we obtain the desired integrability conditions. The fact that the resulting ME coupling is mutually cobounded follows from \cref{prop:fundamentalDomainBis} together with the fact that $\uu_n$  is
injective, and has $2L$-dense image.
\end{proof}

\begin{proof}[Proof of \cref{thm:SolvableExpME}]
By Proposition 2.30 of \cite{DKLMT-22}, we know that the existence of such a measure equivalence is an equivalence relation. 
Then, by \cite[Theorem 8.1]{DKLMT-22}, we have that $(\Z/k\Z)\wr\Z$ and $\operatorname{BS}(1,k)$ are in the same equivalence class for any $k$.
By \cref{thm:example}, we have that $(\Z/k\Z)\wr\Z$ lies in the equivalence class of $\SOL_A$. The groups $\SOL_A$, for different matrices $A$, are all equivalent since they are uniform lattices in the same connected Lie group. So the theorem is proved.
\end{proof}

\bibliographystyle{alpha}
\bibliography{bib}
\end{document}